\documentclass[a4paper]{article}
\usepackage[usenames,dvipsnames]{xcolor}

\usepackage[utf8]{inputenc}
\usepackage[english]{babel}
\usepackage{verbatim}
\usepackage{svg}
\usepackage{a4wide}
\usepackage{subcaption}

\usepackage{float}

\usepackage{amsmath,amsthm,amsfonts,amssymb}
\usepackage{graphicx}
\usepackage{amsfonts}

\usepackage{amsmath,amsthm,amsfonts}
\usepackage{graphicx}
\usepackage[colorlinks=true, allcolors=blue]{hyperref}
\usepackage{amsfonts}

\usepackage[showonlyrefs]{mathtools}

\usepackage{enumerate} %% 

\usepackage{titlesec}

\usepackage[
backend=biber,
sorting=ynt
]{biblatex}
\newcommand{\ROT}[1]{\begin{bmatrix} 1 & 0 & 0 \\
0 & \cos{#1} & -\sin{#1} \\
0 & \sin{#1} & \cos{#1} \end{bmatrix}}

\newcommand{\D}{\mathcal{L}}
\newcommand{\R}{\mathbb{R}}

\newcommand{\SetXdelta}{\mathcal{X}_{{\delta},{\eta}}}

\newcommand{\A}[1]{\Gamma_{#1}^{0}(v_0,\rho)}
\newcommand{\AExpress}[1]{\Gamma_{#1}^{0}}

\newcommand{\B}[1]{\Gamma_{#1}^{1}(v_0,\rho)}
\newcommand{\BExpress}[1]{\Gamma_{#1}^{1}}

\newcommand{\CX}[1]{\int_{\Omega} r^2 \cos(2\theta)^2\sigma^{p}(v_0 + r\rho \cos(2 \theta)     )  \ \frac{P(r)}{\pi}  d\theta dr}
\newcommand{\IT}{I_{1:2}^{\top}}
\newcommand{\I}{I_{1:2}}

\newcommand{\Teinv}{\mathfrak{T}_t}
\DeclareMathOperator{\Sat}{cut}

\newcommand{\tun}{t_1}
\newcommand{\tdeux}{t_2}

\newtheorem{thm}{Theorem}
\numberwithin{thm}{section}
\newtheorem{Proposition}[thm]{Proposition}
\newtheorem{Lemma}[thm]{Lemma}

\newtheorem{Assumption}[thm]{Assumption}
\theoremstyle{remark}
\newtheorem{remark}[thm]{Remark}

\usepackage[normalem]{ulem}

\title{Activity estimation via distributed measurements in an orientation sensitive neural fields model of the visual cortex\thanks{This work has been partially supported by the ANR-20-CE48-0003.}}

\author{Adel Malik Annabi\footnote{Université Côte d’Azur, Inria, CNRS, LJAD, France
(email: {\tt adel-malik.annabi@inria.fr, jean-baptiste.pomet@inria.fr, ludovic.sacchelli@inria.fr})}
, Jean-Baptiste Pomet\footnotemark[2],
Dario Prandi\footnote{Université Paris-Saclay, CNRS, CentraleSupélec, Laboratoire des Signaux et Systèmes, 91190, Gif-sur-Yvette, France (email: {\tt dario.prandi@centralesupelec.fr})},
Ludovic Sacchelli\footnotemark[2]}

\begin{document}

\maketitle

\begin{abstract}

This paper investigates the online estimation of neural activity within the primary visual cortex (V1) in the framework of observability theory.
We focus on a low-dimensional neural fields modeling hypercolumnar activity to describe activity in V1. We utilize the average cortical activity over V1 as measurement. Our contributions include detailing the model's observability singularities and developing a hybrid high-gain observer that achieves, under specific excitation conditions, practical convergence while maintaining asymptotic convergence in cases of biological relevance. The study emphasizes the intrinsic link between the model's non-linear nature and its observability. 
We also present numerical experiments highlighting the different properties of the observer.

\end{abstract}

\section{Introduction} 

Having techniques to estimate online the neural activity of certain cortical areas has many potential uses. 
This can range from monitoring and treating brain health \cite{sood2016nirsParameterEst}
to applications in neuroscience, such as 
psychology \cite{astolfi2009estimation},
brain-computer interfacing \cite{babiloni2007estimationBMI},
or 
the analysis of specific neural phenomena such as visual illusions \cite{gulbinaite2017triple}. 
Recently, this has opened the door to the design of control theory inspired techniques, feedback-loop control such as in the fields brain-machine interface \cite{sorrell2021brainBMI} and deep brain stimulation \cite{herron2017cortical}. For instance, stabilization of signals in the brain via feedback control can be used to alleviate Parkinson's disease symptoms \cite{detorakis2015closed}. 
In most practical cases, neural activity can only be partially measured. It may then be crucial to be able to provide efficient and reliable online estimation methods based on measurements \cite{hu2010kalman,astolfi2009estimation,hartoyo2019parameter}.

In the present study, we focus on models where the neuronal activity is a distribution over space evolving in time according to well known neural fields equations. 
These models rely on 
integro-differential equations obtained by averaging the activity of large groups of neurons in the brain. They were introduced in the seminal works \cite{wilson1973mathematical,amari1977dynamics} and provide a powerful theoretical framework for studying brain activity. See \cite{breakspear2017dynamic} for comparison of different large scale models of neuronal activity, \cite{terry2022neural} for a review of neural fields models and \cite{bressloff2011spatiotemporal,coombes2014neural} for a more in depth analysis of these equations. 
These models have a rich history of application, particularly in the study of the primary visual cortex V1 \cite{blumenfeld2006neural,ben1995theory,bressloff2011spatiotemporal}, and in particular for explaining visual illusions \cite{bertalmio2020visual,bertalmio2021cortical,tamekue2024reproducibility,veltz2010illusions}. They have also demonstrated the capability to replicate numerous phenomena observed in experimental data, often acquired through voltage-sensitive dye (VSD) imaging, see \cite{blumenfeld2006neural,markounikau2010dynamic,o2018interhemispheric} for example. 

The focus of this paper is on state estimation for a low dimensional model of V1 introduced in \cite{blumenfeld2006neural}. This particular model incorporates many characteristic properties of the \emph{ring model} introduced in \cite{ben1995theory}  and demonstrates a noteworthy ability to accommodate qualitatively for experimental data, on orientation and selectivity, as shown in \cite{o2018interhemispheric}.
To reflect the limitations of standard physical measurement in the cortex, as those obtained via electrodes, we consider distributed measurements and in particular the average voltage within the area of interest.

To achieve our state estimation objective,
we approach neural fields models as  input-output systems.  This perspective allows us to leverage the formalism of control theory, particularly the concept of \emph{observers}. Observers, in this context, are dynamical systems designed to provide online estimates online of the state. 
See \cite{bernard2022observer} for a review on observer theory for continuous dynamical system.
From the standpoint of observability, the considered model introduces distinct challenges as a non-linear system. 
Indeed, the system may not be observable at all times, due to the state crossing some singular regions. 
We propose an hybrid observer design which accounts for these difficulties.
This approach allows for practical state estimation in regions where observability is hindered, and guarantees tunable exponential asymptotic convergence for cases of biological relevance.

The application to neural fields model of observability theory is quite recent. Up to our knowledge, the only study in this direction is \cite{brivadis2022online} which proposes an adaptive observer for a certain neural fields equation. 
In that paper, the authors consider two interconnected cortical zones, one of which is thoroughly measured. Their observer design is then based on the contractivity of the dynamics of the un-measured zone and a persistence of excitation condition.

The paper is structured as follows. In Section~\ref{sec:models} we present the neural fields model of V1 that is the focus of the observer study and establish the precise 3-D model used in the sequel.
In Section~\ref{Section:Statement of the results} we present the results we obtain in the paper, first on observability properties of the system, then on our tunable observer, whose precise design is to be found in Sections \ref{Section: Observability Mapping} and \ref{Section 4}. The technical part of the paper is divided into three sections. First, Section~\ref{Section Prelemenaries} contains technical preliminaries. Then in Section~\ref{Section: Observability Mapping} we prove the main results regarding observability and give the construction of a pseudo-inverse needed in the observer. Finally in Section~\ref{Section 4} we prove the main properties of the observer. In the last section, Section~\ref{Section:Numerical Simulations}, we discuss a numerical simulation of the observer.
 
\section{Modelisation} \label{sec:models}

In this section, following the steps in \cite{blumenfeld2006neural}, we present a neural field models of V1 and the necessary steps to obtain the low dimensional model on which the observer is constructed.

\subsection{General model on V1} 
The goal of the present study is to estimate the post-membrane potential 
over the visual cortex V1 represented by a bounded open set $\hat{\Omega}\subset \R^2$. We denote
by $V(x,t)\in \R$ the average post-membrane potential at a point $x\in \hat{\Omega}$ at a time $t$.
The evolution of $V$ is modeled by a single-layer neural field equation (see \textit{e.g.}, \cite{terry2022neural})
\begin{equation}
\label{GeneralEquation} 
    \begin{cases}
    \tau \partial_{t}V(x,t)=- V(x,t) +  J\cdot \sigma( V(\cdot,t))(x)  +  I_{\text{ext}}(x,t), &\\
    V(0)\in L^2(\hat\Omega).
    \end{cases}
\end{equation}
Here, $I_{\text{ext}}\in L^{\infty} (\mathbb{R},L^2(\hat\Omega))\cap C^{0}(\mathbb{R},L^2(\hat\Omega))$ denotes the external input, 
$\tau>0$  the temporal synaptic constant. The operator $J$ models the neuronal interaction and is given by an integral kernel $j(\cdot,\cdot)\in L^\infty(\hat\Omega\times \hat\Omega)$ through the expression
\begin{equation} \label{Introduction J}
 J\cdot u (x) = \int_{\hat{\Omega}} j(x,y)u(y)dy, \qquad u\in L^2(\hat\Omega).
\end{equation}
The firing rate function $\sigma:\mathbb{R}\to\mathbb{R}$ models how a population ``charges'' or ``discharges'' and is typically chosen to be a smooth sigmoidal function. 
Equation~\eqref{GeneralEquation} is well defined and has a unique global solution for each initial condition $V(\cdot, 0) \in  L^2({\hat{\Omega}})$ which has the same regularity as $I_{\text{ext}}$. See \cite{bressloff2011spatiotemporal,veltz2010local} for instance.

In this work we will assume\footnote{See Appendix~\ref{Appendix:Modelisation adjustement.}, for a discussion on the adjustments required to adapt the ensuing ensuing analysis to the case of strictly positive sigmoids and/or the presence of thresholds $h_0$, i.e., $\sigma(\cdot-h_0)$.} $\sigma$ to satisfy the following.
\begin{Assumption}
\label{ass:sigma}
The function $\sigma$ is a $C^{\infty}(\R,\R)$ function that is odd, strictly increasing (i.e., $\sigma'>0$), convex-concave (i.e.,  $x\sigma''(x)\leq 0$ for all $x\in \R$), and such that
\begin{equation} \label{Assumption Sigma}
    \lim\limits_{\substack{x \to -\infty}} \sigma(x)=  -1, \quad
    \sigma(0) = 0, \quad
    \lim\limits_{\substack{x \to +\infty}} \sigma(x)=  1, \quad 
    % \sigma'(x)>0 \quad \forall x\in \R,  
    \quad \max_{\R} \sigma'=\sigma'(0).
\end{equation}    
\end{Assumption}

We are interested in the observability and observer design problem for solutions of \eqref{GeneralEquation}.
As mentioned, we consider as measurement the cortical activity averaged over the whole domain, as represented by the function 
\begin{equation}
    \label{eq:measurement-h}
    h(V(\cdot,t))=\int_{ \hat{\Omega}} V(x,t) \,dx.
\end{equation}

In the context of our study, the observability problem pertains to determining whether the internal state of a neural field, described by \eqref{GeneralEquation}, can be fully inferred from the cortical activity measurements averaged over the domain, as defined by \eqref{eq:measurement-h}. Formally, this problem explores the feasibility of reconstructing the neural field's dynamic state $ V(x,t)$ solely from the output $y(t)$. The observer design problem involves creating a system to estimate the internal state of a model based on the output $y$.

The non-linearity of the sigmoid function $\sigma$ plays a crucial role in the observability of the system. In the hypothetical case where $\sigma$ was linear, the dynamics of $ t \mapsto h \circ V(\cdot, t) $ will feature only the compositions of $ h \circ V(\cdot, t) $ and the system's inherent constants.
rendering Equation~\eqref{GeneralEquation} non-observable. It is precisely the non-linear characteristics of $\sigma$ that enable us to differentiate between potential solutions.

\begin{remark}
    Equation \eqref{GeneralEquation} is what is called in the literature a voltage based model, whereas the original models in \cite{ben1995theory} and \cite{blumenfeld2006neural} are activity based models.
    We find this formulation more convenient for mathematical analysis. 
    Moreover, as we can always transition from an activity based model to a voltage-based one via an adequate change of variables, our results can be applied to the former. See Appendix~\ref{Appendix:Modelisation adjustement.}.
\end{remark}

\subsection{Finite-dimensional reduction in V1}
From~\cite{blumenfeld2006neural}, since the experimental solutions tends to be unimodal, we follow the steps introduced in the article to obtain a reduced model. 
We are interested in two features encoded by neurons in V1: orientation $\theta \in [-\frac{\pi}{2}, \frac{\pi}{2})$ and selectivity preference $r \in [0,\infty)$  \cite{hubel1962receptive}.
Mainly, we assume that there exists a polar mapping $\Phi:x  \in \hat\Omega \mapsto (r_x,\theta_x)\in \Omega:=[0,+\infty)\times [-\pi/2,\pi/2)$. See Figure~\ref{fig:SelectivityOrientation}. Henceforth, we will work on \eqref{GeneralEquation} in the coordinates induced by $\Phi$.
 
\begin{figure}[t]
    \centering
    \begin{subfigure}[t]{0.45\textwidth}
        \vphantom{\includegraphics[width=\textwidth]{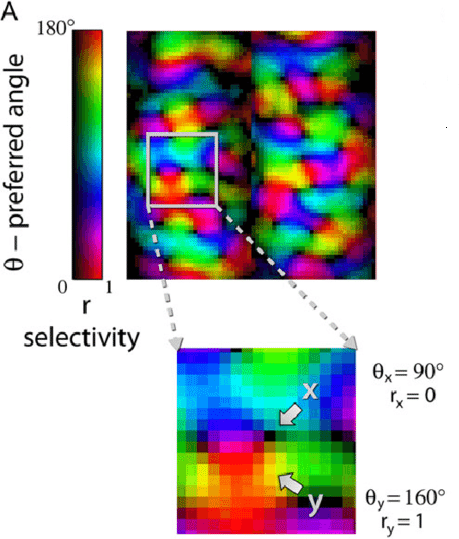}}
        \includegraphics[width=\textwidth ]{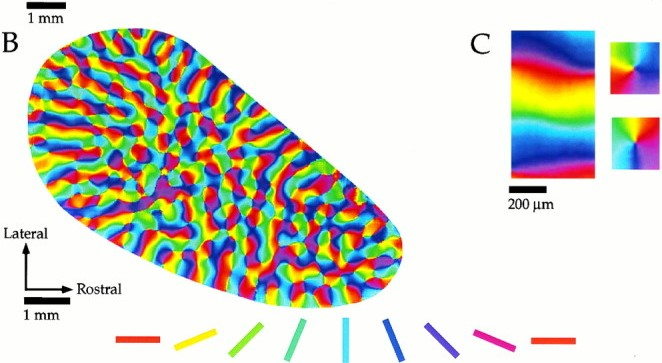}
        
        \caption{Comprehensive orientation preference (OP) map in V1 visualized through optical imaging, depicting the spatial distribution of orientation-selective neurons.
        See \cite{bosking1997orientation}.}
        \label{fig:PolarMapping}
    \end{subfigure}
    \hspace{.05\textwidth}
    \begin{subfigure}[t]{0.4\textwidth}
        \includegraphics[width=\textwidth]{PolarMappingSO}
        \caption{Orientation preference and selectivity mapping in the visual cortex V1, showing angular preference (top) and selectivity index (bottom) with key orientation angles highlighted. Here the values of the selectivity have been normalized. See \cite{o2018interhemispheric}.}
        \label{fig:SelectivityOrientation}
    \end{subfigure}

      \caption{Visual cortical mapping of orientation preference and selectivity in V1 region: (a) illustrates a comprehensive view of orientation-selective neuronal distribution through optical imaging, and (b) presents a detailed analysis of orientation preference and selectivity, with emphasis on the normalized selectivity index and key angular orientations.}
    \label{fig:combinedFigure}
\end{figure}

The polar mapping $\Phi$ induces a change of measure: 
\[
dx=P(r,\theta)drd\theta,
\]
where $P(r,\theta)= c\, |\!\det (\partial_{r,\theta} \Phi^{-1}) |/|\hat\Omega|$ and $c>0$ is a normalization constant making $P$ a probability density. Following \cite{ben1995theory} and \cite{blumenfeld2006neural}, we suppose that the distribution over orientation preference is uniform, i.e that cortex handle all orientations similarly, obtaining then that $P(r,\theta)= \frac{P(r)}{\pi}$. 
Since neuronal selectivity ranges in a finite interval, we henceforth assume $P$ to have compact support.

 Following the steps in \cite{blumenfeld2006neural} and \cite{veltz2010local},  we look for an operator $J$ in~\eqref{Introduction J} such that the experimentally observed cortical stationary states at rest ($I_{\text{ext}}=0)$ are in its range. Since these stationary states tends to be unimodal i.e. a single dominant fourier mode, we can then well approximate the connectivity kernel by :
\begin{equation}
    j(x,y)=J_0 + J_1 r_x r_y \cos(2(\theta_x-\theta_y)), \quad J_1 > 0, \quad  J_0 \neq 0.
\end{equation}
See~\cite{blumenfeld2006neural} and \cite{veltz2011nonlinear}[Chapter 10].
Here, the parameters $J_0$ and $J_1$ encode, respectively, the strength of global inhibition or excitation and the strength of the connectivity.
As $J$ is a compact self-adjoint operator, using that for any $U(\cdot,t) \in L^2(\hat \Omega)$, $U=U^{\parallel} + {U}^{\perp}$, where $U^{\parallel} \in \text{range}( J)$, Equation~\eqref{GeneralEquation} splits into
\begin{equation}
\left\{
 \begin{array}{l}
\tau  \dot{V}^{\parallel}=- V^{\parallel} + J\cdot \sigma(V^{\parallel} + V^{\perp}) \ + I^{\parallel} ,     \\
  \tau \dot{V}^{\perp}=- V^{\perp} + I^{\perp},
\end{array}\right.
\end{equation}
with
\begin{equation}\label{eq: V par}
V^{\parallel}(r_x,\theta_x,t)=v_0(t) + v_1(t) r_x \cos(2\theta_x) + v_2(t) r_x \sin(2\theta_x).    
\end{equation}
Setting the input such that $I^{\perp}$ is negligible, $V^{\perp}$ converges exponentially towards 0. We henceforth assume $I^{\perp}=0$ and set $V=V^{\parallel}$.

Finally, observe that the measurement $h$ is recast in the following form
    \begin{equation}
    h(V(\cdot,t))= v_0(t),
    \end{equation}

\begin{remark}
    Observe that we can also obtain the same measurement by a well placed electrode centered around a pinwheel, i.e. any point $x_0\in \R^2$ that verifies $r_{x_0}=0$. Indeed we have using \eqref{eq: V par}
    \begin{equation}
     \langle V(\cdot,t), \delta_{x_0} \rangle=v_0(t).
    \end{equation}
\end{remark}

\subsection{Main model of study}

The observations of the previous section allow us to reduce the infinite dimensional equation \eqref{GeneralEquation} to the following finite dimensional ODE for $v=(v_0,v_1,v_2)\in \mathbb{R}^3$ 
\begin{equation}\label{EqP}  \left\{
 \begin{array}{ll}
    \tau \dot{v}_{0} \  = \  -  v_{0} \ +\  J_0\int_{\Omega} \sigma( V(r,\theta,v)  )  \frac{P(r)}{\pi} d\theta dr \ +\  I_0(t)&:=f_0(v,I(t)),  \\
    \tau \dot{v}_{1} \ = \ - v_{1} \ + \  J_1\int_{\Omega} r \cos(2\theta)\sigma( V(r,\theta,v)    )  \ \frac{P(r)}{\pi}  d\theta dr  +\ I_{1}(t)&:=f_1(v,I(t)), \\
   \tau  \dot{v}_{2} \ = \ - v_{2} \ +  \   J_1 \int_{\Omega}  r\sin(2\theta)\sigma( V(r,\theta,v)   )  \ \frac{P(r)}{\pi}  d\theta dr  + \  I_{2}(t)&:=f_2(v,I(t)).
 \end{array}\right.
\end{equation}
Here, $J_0$ and $J_1$ are real non-zero parameters of the system, $I=(I_0,I_1,I_2)$ is the external input, which we assume to be at least continuous,
$P(\cdot)$ is a compactly supported probability density over $[0,+\infty)$, $\Omega = [0,+\infty)\times [-\pi/2,\pi/2)$, and $V(r,\theta,v)$ is the neuronal activity given by
\begin{equation}                        
V(r,\theta,v)\ = \ v_0  \ + \ rv_1\cos(2\theta) \ + \ rv_2 \sin(2\theta).
\end{equation}
The ouptut is, at each time, the space average of this neuronal activity, hence $y=h\circ v$ with $h:\R^3\to\R$ defined by
\begin{equation}\label{EqP-y}
 h(v_0.v_1,v_2)=v_0\,,\quad\text{i.e.}\ y=v_0\,.
\end{equation}

As $\tau >0$ does not interfere in our study of the observability, we set $\tau=1$ from here to ease the proofs.

Concerning the existence and the maximal bound of the solutions, we have the following proposition, the proof of which is found in Section~\ref{Section Prelemenaries}
\begin{Proposition} \label{Solution defined Bounded}
Assume $I\in C^0([0,+\infty))$. The dynamics \eqref{EqP} admit a unique global solution $v$ defined on $\R$ for every initial condition $v(0)\in \R^3 $. Moreover, if $\sup_t ||I|| < \infty$, then there exist $R^* > 0$, such that for any $R \geq R^*$, $B_{\R^3}(0,R)$ is an invariant attracting set by the dynamic $f$.
\end{Proposition}

\begin{remark}
    In the case where $P$ is a Dirac mass,
    \eqref{EqP} reduces to the ring model introduced in \cite{ben1995theory}. 
    Observe also that the assumption of compact support for $P$ is not essential. Indeed, our analysis applies as soon as $P$ admits moments up to the second order.
\end{remark}

\paragraph{Notation.}
We denote any $X\in\R^3$ as $X=(X_0,X_1,X_2)^\top$. Moreover, we let $X_{1:2}\in \R^2$ be the vector composed of the two last elements, i.e., $X_{1:2}=(X_1,X_2)^\top$.  We denote also by $f(v,I(t))=(f_0(v,I(t)),f_1(v,I(t)),f_2(v,I(t)))$ the right hand side of equation~\eqref{EqP}.

\section{Statement of the results} \label{Section:Statement of the results}

In this section, we present the results achieved in the paper. In a first step we discuss the dynamical properties of the system in connection with observability of the state. Then we propose an estimation method and present its convergence properties under different assumptions. 

\subsection{Observability} 
\label{sec:observability}

Here, we investigate the observability properties of \eqref{EqP}-\eqref{EqP-y}.
We first consider the simplest notion of observability, namely to what extent an output $t\mapsto y(t)$, for a prescribed input $t\mapsto I(t)$, can be produced by only one solution $t\mapsto v(t)$.
This is called (instantaneous) observability in 
\cite[Section~3.1.2]{bernard2022observer} or \cite[Chap.\ 2, Defn.\ 1.2]{gauthier_kupka_2001}.

The following preliminary proposition states that observability \emph{does not} hold when $v_0$ is zero and gives, away from $\{v_0=0\}$, the ``maximum information'' that can be extracted from the knowledge of the output $y(.)$ without any assumption on the input $I(\cdot)$. 

\begin{Proposition} \label{Observability rho}
We have the following.
\begin{enumerate}

\item[(i)] Let  $I\in C^0([0,\infty), \mathbb{R}^3)$ such that $I_0=0$. The output $h\circ v=v_0$ of any solution $v$ of \eqref{EqP} with initial condition $v(0)=(0,v_1(0),v_2(0) )^{\top} \in \R^3 $ is identically null. In this particular case, there exist $v$ and $\hat v$ solutions of \eqref{EqP} such that
\begin{equation}
    h(v(t))=h(\hat v(t)) \ \text{and} \ v(t) \neq \hat{v}(t), \ \forall t \ge 0. 
\end{equation}

\item[(ii)] Let $v$ and $\tilde v$ be two solutions of \eqref{EqP} such that $h(v(t)) = h (\tilde v(t))$ for all $t$ in an interval $[0,t^*]$, for some $t^*>0$. Then, for any time $t'\in[0,t^*]$ such that $h(v(t'))\neq0$ we have
\begin{equation}
    \label{obs-faible}
    |v_{1:2}(t')|=|\tilde{v}_{1:2}(t')|
    \ \text{and}\ 
    v(t')_{1:2}^{\top} I_{1:2}(t') = \tilde v(t')_{1:2}^{\top} I_{1:2}(t').
\end{equation}

\end{enumerate}

\end{Proposition}

Let us now state two more precise results that elaborate on Points~\textit{(i)} and \textit{(ii)} above, starting with
Point~\textit{(i)}, that tells us that observability does not hold in the set $\mathcal{Z}_0 := \{ v \in \R^3:\: v_0 = 0\}$, and that $\mathcal{Z}_0$ is invariant by the dynamics if the first component of the input is identically zero.
Importantly, as the output under consideration is $v_0$, the value of the output tells us how close the state $v$ is from $\mathcal{Z}_0$.
Since we need to build an observer that tolerates trajectories passing through $\mathcal{Z}_0$, Proposition~\ref{Proposition : Passage} below, proved in Section~\ref{Section 3.2}, states that, under the assumption that the first component of the input is bounded from below, trajectories may pass through $\mathcal{Z}_0$ only once and gives an estimate of the time a solution passes ``close to'' that set.
Note that this assumption is often made in the literature, see \textit{e.g.} \cite{veltz2010local}.

\begin{Proposition}
\label{Proposition : Passage}
Consider an input $I:[0,+\infty)\to\R^3$ such that $I_0(t)\ge c>0$ for all $t \geq 0$, and a solution $v(\cdot)$ of \eqref{EqP} associated to this input.
Let 
\begin{equation}\label{eq:deltastar}
    \delta^*=\frac{c}{1+|J_0|\sigma'(0)},
\end{equation}
and fix some $\delta$, $0<\delta<\delta^*$.
The first component $v_0(t)$ of the solution vanishes at most one time, and the set of times $t$ such that 
$|v_0(t)|\leq\delta$ is a time-interval of length no 
larger than $t_\delta$ given by
\begin{equation}\label{eq:tdelta}
 t_\delta= \frac{\delta^*}c \frac{2\delta}{\delta^*-\delta}\,.
\end{equation}
  Moreover, if there exist a time $t^* \geq 0$ such that $v_0(t^*) \geq \delta$, then $v(t) \geq \delta$ for all times $t\geq t^*$.
\end{Proposition}

\begin{remark}
    If the parameter $J_0$ is negative and $I_0(t)>c>0$ the flow of $v_0$ is always positive when $v_0$ takes negative values. The proposition above is there solely to treat the case $J_0>0$. 
\end{remark}

Point~\textit{(ii)} in Proposition~\ref{Observability rho} states partial observability; it is completed by Proposition~\ref{Observability Theorem} below under proper conditions on $I_{1:2}$.
\begin{Proposition}\label{Observability Theorem}
    Consider an input $I:[0,+\infty)\to\R^3$ such that $I_0(t)\ge c>0$ for all $t \geq 0$.
    For any $t^*>0$, the two following assertions are equivalent
    \begin{enumerate}[(i)]
        \item For any two solutions $v$ and $\tilde{v}$ of \eqref{EqP} associated to this input, 
        \begin{equation}
            h(v(t))=h(\tilde{v}(t)) \quad \forall t\in [0,t^*] \quad \implies \quad  v(t)=\tilde{v}(t) \quad \forall t\in [0,\infty).
        \end{equation}
        \item  There exists $t\in [0,t^*]$ such that $I_{1:2}(t) \wedge \dot{I}_{1:2}(t)\neq 0$.
    \end{enumerate}
Here, $a\wedge b$ denotes the determinant of the matrix $[a;b]$ for any vector $a,b$ of $\R^2$
\end{Proposition}

The proof of the above results relies on considerations regarding a stronger form of observability on our model, often called \emph{differential obervability} \cite{bernard2022observer,gauthier_kupka_2001}, that requires that the value of a certain number of time-derivatives of the output $y=h(v(t))$ at some time $t$ uniquely determines the value of the state $v(t)$ at the same time.
The time derivative is given by the differential operator $\D$, also called ``total derivative'' along the time-varying differential equation \eqref{EqP} (see \textit{e.g.}, \cite{bernard2022observer}), that maps a real-valued function $g \in C^{k}(\mathbb{R}^3\times \mathbb{R})$  to 
$ \D g \in C^{k-1}(\mathbb{R}^3\times \mathbb{R}, \R)$ defined by
\begin{equation}  \label{eq: operator differentiel D}
\D g(v,t)=f_0(v,I(t))\partial_{v_0}g(v,t)+f_1(v,I(t))\partial_{v_1}g(v,t)+f_2(v,I(t))\partial_{v_2}g(v,t)+\partial_{t}g (v,t).
\end{equation}
Note that the input $I(\cdot)$ is fixed and sufficiently smooth; a different $I(\cdot)$ defines a different differential operator $\D$.
Hinting that the number of time-derivatives needed is no larger than three, we study the time-dependant mapping $v\mapsto T_t(v)$ defined, for all $t$ in $[0,+\infty)$, by
\begin{equation}\label{eq:Introduction T}
T_t(v)=
\begin{pmatrix}
h(v) \\ \D h\,(v,t) \\ \D^2h\,(v,t) \\ \D^3h\,(v,t)
\end{pmatrix},
\end{equation}
The map $T_t$ is often called the observability mapping at time $t$, indeed its injectivity implies the aforementioned differential observability.
The following proposition is our main result on differential observability of \eqref{EqP}-\eqref{EqP-y}. It is proved in Section~\ref{Section 3.2}.
Define the sets
\begin{equation}
    \label{setsZ}
    \mathcal{Z}_0=\{v \in \R^3 : v_0=0 \},\quad
    \mathcal{Z}_{1:2}=\{v\in \R^3\mid v_1=v_2=0\},\quad
    \mathcal{Z}:=\mathcal{Z}_{1:2} \cup \mathcal{Z}_0\,.
\end{equation}
\begin{thm}
    \label{prop:T-inj}
    Assume $I\in C^2([0,+\infty),\R^3)$. Let $t\ge 0$ be such that $I_{1:2}(t)\wedge \dot I_{1:2}(t)\neq 0$. \begin{itemize}
        \item The map $T_t :\R^3 \setminus \mathcal{Z}_0\to\R^4$ is an injection.
        \item The map $T_t:\R^3 \setminus \mathcal{Z}\to\R^4$ is an injective immersion.
    \end{itemize}
\end{thm}
The loss of injectivity in the singular region $\mathcal{Z}_0$ is related to the loss of observability in Proposition~\ref{Observability rho}.
The first point implies that the restriction of $T_t$ to $\R^3 \setminus \mathcal{Z}_0$ has an inverse $T_t(\R^3 \setminus \mathcal{Z}_0)\to\R^3 \setminus \mathcal{Z}_0$; the second point implies that this inverse is differentiable on $T_t(\R^3 \setminus \mathcal{Z})$.
In fact this inverse is continuous, but fails to be locally Lipschitz continuous at images of points in $\mathcal{Z}_{1:2}$. 

\begin{remark} \label{remark: Biological Interest.}
    Model~\eqref{EqP} has been introduced in \cite{blumenfeld2006neural} to accommodate for experimental solutions obtained through VSD. These solutions always verify that
    \begin{equation}\label{eq: Solution of interest}
        \exists \eta^*>0,t_{\eta^*}>0, \quad |v_{1:2}| \geq \eta^*, \ \forall \ t \geq t_{\eta^*}.
    \end{equation}  
    In the next section, we discuss a tunable observer that verifies exponential asymptotic convergence for solutions verifying~\eqref{eq: Solution of interest}, while keeping some form of practical convergence for solutions that do not. We observe that appropriately chosen inputs and parameters may produce solutions that cross the region~$\mathcal{Z}_{1:2}$ for arbitrarily large times.
\end{remark}

In light of the above discussion, we define the following family of tubular neighborhoods of $\mathcal{Z}$, indexed by two parameters $\delta>0$, $\eta>0$:
\begin{equation}\label{eq:001}
    \mathcal{Z}_{\delta,\eta}:=\{ v \in \R^3:\: |v_0|<\delta \} \cup \{ v \in \R^3:\: |v_{1:2}|<\eta \}.
\end{equation}
The restriction of the inverse of $T_t$ to a bounded part of $T_t(\R^3 \setminus \mathcal{Z}_{\delta,\eta})$ will have a global Lipschitz constant.

\medskip

Let us now derive from the above observability conditions an assumption that will be useful for designing an observer in next subsection.
We saw in Proposition~\ref{Proposition : Passage} that the trajectories cross the non observalility locus $\{v_0=0\}$ at most once if $I_0$ stays away from zero and in Proposition~\ref{Observability Theorem} that observability occurs when $v_0$ and $I_{1:2} \wedge \dot{I}_{1:2}$ stay away from zero; in view of these points, we make the following assumption on the input, for the rest of the paper.
\begin{Assumption} \label{Strong Persistence}
The input $t\mapsto I(t)$ is in $C^3([0,+\infty),\mathbb{R}^3)$ and is bounded on $[0,+\infty)$,  as well as 
its time-derivatives of order up to three, and the following lower bounds hold for some positive numbers $c$ and $\mu$:
\begin{equation}
    \label{eq:lower-bound+cond-det-I}
    I_0(t) \ge c>0 
    \quad \text{and} \quad
    | I_{1:2}(t) \wedge \dot{I}_{1:2}(t) | \geq \mu>0
    \qquad \text{for all $t$ in $[0,+\infty)$.}
\end{equation}
\end{Assumption}

\begin{remark}
    The same results in this paper can be obtained by assuming $|I_0(t)|\geq c$ rather than $I_0(t)>c$. The positive sign of $I_0$  in the first relation in \eqref{eq:lower-bound+cond-det-I} is chosen to facilitate the exposition. 
\end{remark}

\begin{remark}
    The second condition in \eqref{eq:lower-bound+cond-det-I} is a stricter version of the standard lower bound for persistence of excitation-type conditions
    \begin{equation}\label{persist1}
      \exists \mu^*, t^* >0, \ \forall t_0>0, \   \frac{1}{t^*}\int_{t_0}^{t_0+t^*} I_{1:2}(t)^{\top} I_{1:2}(t)dt \ \succeq \ \mu^* \operatorname{Id}_{2\times2},
    \end{equation}
where $\operatorname{Id}_{2\times2}$ denotes the identity matrix of dimension 2.
\end{remark}

\subsection{Observer design}

In this section, we propose an observer based on the high-gain construction \cite{bernard2022observer,gauthier_kupka_2001} that takes into account the observability difficulties encountered in the model.
Using the notations of previous section, the mapping $(t,v)\mapsto T_t(v)$
is structured to satisfy the following condition for all $v \in \R^3$ and $t \in \mathbb{R}$:
\begin{equation} \label{eq:Immersed space}
\partial_v T_t (v,t) f(v,I(t)) + \partial_{t}T(v,t)=AT_t(v,t) \ + \ W \D^4 h (v,t),
\end{equation}
where $A$ and $W$ are matrices defined as:
\begin{equation}
A=
\begin{pmatrix}
0&1&0&0 \\ 0&0&1&0\\0&0&0&1\\0&0&0&0
\end{pmatrix},
\quad
W=
\begin{pmatrix}
0\\0\\0\\1
\end{pmatrix}.
\end{equation}
The next step involves defining a Lipschitz pseudo-inverse $\mathfrak{T}_t$ for $T_t$ at any given time $t\ge0$. This pseudo-inverse is essential for extracting $v(t)$ from the values $z(t) = T_t(v(t))$ and to make so that only terms depending on $z$ appears in the right-hand side of \eqref{eq:Immersed space}. A high-gain can then be designed which estimates the value of $z(t)=T_t(v(t))$ for any solutions $v$ of \eqref{EqP} based on its output $y$. The efficacy of this design hinges on existence of a pseudo-inverse of $T_t$ which is Lipschitz, ensuring the observer’s accurate and stable convergence.
For this design to be feasible,some form of differential observability is required as mentioned in the previous section. Further details on this construction can be found in~\cite{besanccon2007nonlinear} for example.

 Within the established framework, let us now construct, for each time $t \ge 0$, a globally Lipschitz continuous mapping denoted as $\Teinv$, serving as a left inverse to the observability mapping $T_t$ within the set $\mathcal{Z}_{\delta,\eta}^{c} \cap B_{\R^3}(0,R)$ \footnote{$\mathcal{Z}_{\delta,\eta}^{c}$ denotes the complement of the set $\mathcal{Z}_{\delta,\eta}$.}, where $R$ is a sufficiently large constant, $0 < \eta < R$, and $0<\delta < \delta^*$, where $\delta^*$ is defined in Proposition~\ref{Proposition : Passage}. The validity of Proposition~\ref{Solution defined Bounded} is invoked to affirm the existence of an arbitrarily large $R > 0$, ensuring that $B_{\R^3}(0,R)$ constitutes an invariant attracting set for the system described by equation \eqref{EqP}. Subsequently, the use of a projection map alongside the inherent symmetry of the non-linear term in \eqref{EqP} is employed to broaden the definition of $\Teinv$ across the entire domain $\R^4$ at any time $t \ge 0$. This is encapsulated within the following theorem, the proof of which is given in Section~\ref{Section 3.3}.  See \eqref{eq:t-inverse} for the explicit expression of $\Teinv$.

\begin{thm}
    \label{Theorem: inverse mapping}
     Suppose that Assumption~\ref{Strong Persistence} holds. Let $\delta>0$, $0 <\eta <R$. Then, for any output $y=h\circ v$, there exists a map 
     \begin{equation}
         \mathfrak{T}:\mathbb{R}^4 \times [0,+\infty) \to
         \{(u_0,u_1,u_2)\in \R^3, \delta \leq |u_0| \leq R \ \text{and} \ \eta \leq|u_{1:2}|\leq R^2 \},
     \end{equation}  
     such that $\mathfrak{T}_t = \mathfrak{T}(\cdot,t)$ is globally Lipschitz continuous uniformly with respect to $t$, and satisfies
    \begin{equation}
        \label{eq:estimate-t-inverse}
       | \mathfrak{T}_t(T_t(u)) - u| \leq \  \eta \chi(u), \quad \forall u \in \mathcal{Z}_{\delta,\eta}^{c}\cap B_{\R^3}(0,R) \text{ such that }  u_0y(t)>0, \  |u_0| \geq \delta 
    \end{equation}
where
\begin{equation}\label{E:chi}
\chi(v):=\begin{cases}
        0 & |v_{1:2}| \geq \eta, \\
        1 & \text{otherwise}.
    \end{cases}
\end{equation}
\end{thm}

We propose now the following observer which takes into account the consideration of Remark~\ref{remark: Biological Interest.}.

Let be $0<\delta<\delta^*$ where $\delta^*$ is defined in Proposition~\ref{Proposition : Passage}, $0<\eta<R$, $l \geq 1$. Let $C=(1,0,0,0)\in\R^{1\times 4}$ and $P_{l}(i,j):= \frac{(-1)^{i+j}}{l^{i+j-1}} \cdot \frac{(i + j - 2)!}{(i - 1)!(j - 1)!}$.
Considering $y=h\circ v$, where $v$ is a solution of \eqref{EqP}, we study a tunable state observer, piece-wise defined as follows:
\begin{equation} \label{Observer 1}
    \begin{cases}
    \hat{v}(t)= \Teinv(\hat{z}(t)), \quad
    \dot{\hat{z}}(t)=A\hat{z}(t) + \D^4h (\mathfrak{T}_{t}(\hat{z}(t))) W -P_l^{-1}C^{\top}(C\hat{z}(t)-y(t)), &\ \text{if } |y(t)|>\delta, 
     \\
    \dot{\hat{v}}(t)=f(\hat{v},t),  \quad \ \hat{z}(t)=T_t(\hat{v}(t))  &\ \text{if } |y(t)|<\delta.
    \end{cases}
\end{equation}

This observer is a hybrid system where simple rules, governed by the sign of $|y(t)|-\delta$, switch the dynamics between two time-varying continuous-time dynamical
systems (explicit dependence on time comes from $I$ and
$y$, whose evolution does not depend on the observer), with states $\hat z\in \R^4$ and $\hat v\in \R^3$.  
Without resorting to a general framework for hybrid systems (see \textit{e.g.} \cite{VdSc-Sch00}), let us
make clear what a solution of this system is.

The initial condition is given by some $\hat z(0)\in\R^4$ if $|y(0)|>\delta$ (or $|y(0)|=\delta$
and $|y(t))|>\delta$ for small positive $t$), and by some $\hat v(0)\in\R^3$ if $|y(0)|<\delta$ or $|y(0)|=\delta$
and $|y(t))|<\delta$ for small positive $t$.
We show in Proposition~\ref{prop:observer_defined} that the solutions of the corresponding ODE are well-defined until the
first time where $|y(t)|-\delta$ changes sign.

Proposition \ref{Proposition : Passage} tells us that $|y(t)|-\delta$
changes sign at most twice on $[0,+\infty)$; we have to describe the jumps between
state spaces at such a time $t_\star$ (either $\tun$ or $\tdeux$ from Proposition~\ref{Proposition : Passage}) to be complete.
Let $t_\star$ be the first such time.
If $|y(t)|\geq\delta$ on $[0,t_\star]$, $|y(t)|>\delta$ on some $[t_\star-\varepsilon,t_\star)$, the
solution $\hat z(t)$ of the first line of \eqref{Observer 1} is well defined on some
$[0,t_\star+\varepsilon]$ so that it has a limit
at $t=t_\star$, that we may denote by $\hat{z}(t_\star^-)$; the
solution to the hybrid system \eqref{Observer 1} is continued by initializing
the second differential equation with state $\hat v$ according to $\hat{v}(t_\star^-) = \Teinv(\hat{z}(t_\star^-))$ on the next interval where
$|y(t)|-\delta$ is negative;
similarly, if $|y(t)|<\delta$ on some $[t_\star-\varepsilon,t_\star)$ (and $|y(t)|>\delta$
on some $(t_\star,t_\star+\varepsilon]$), the initialization in the new state space is made
according to $\hat{z}(t_\star^-) = T_t(\hat{v}(t_\star^-))$.

Henceforth, with a slight abuse of language, we refer to $\hat v$ as the solution of \eqref{Observer 1}.
We can now state the following, which is proven in Section~\ref{Section 4}

\begin{Proposition}
    \label{prop:observer_defined}
    For any $\delta<\delta^*$, with $\delta^*$  given by \eqref{eq:deltastar} in Proposition~\ref{Proposition : Passage}, any  $0<\eta <R$, solutions of \eqref{Observer 1} are defined globally, piece-wise continuous, with $\hat v$ having at most a single jump.
\end{Proposition}

\begin{thm} \label{Theorem Observer 1}
    Suppose Assumption~\ref{Strong Persistence} holds. For any $R>R^*$, where $R^*$ is defined in Proposition~\ref{Solution defined Bounded}, for any $v$ solution of \eqref{EqP} with output $y=v_0$, and such that $v(0)\in B_{\R^3}(0,R)$, 
    we have the following.
    
    \begin{enumerate}[(i)]
        \item\label{it:main1} For any $\delta<\delta^*$, where $\delta^*$ is defined by \eqref{eq:deltastar} in Proposition~\ref{Proposition : Passage},  $0<\eta<R$, there exist $l^*$, $M>0$, $L>0$ such that for any $l \geq 1$, any interval $[t_0,t_f]$ where $|y(t)| \ge \delta$ for $ t\in [t_0,t_f] $, we have for any solution $\hat v$ of \eqref{Observer 1} with initial conditions $\hat{v}(0) \in \R^3$ and $\hat z(0)= T_0(\hat v(0))$, it holds for any time $t_0\leq t<t_f$
    \begin{equation}\label{eq:bigthm:1}
        |v(t)-\hat{v}(t)| \leq \eta \chi(v(t))  +  M e^{-(l-l^*) (t-t_0)} \biggl( l^3|\hat{z}(t_0)-T_{t_0}(v(t_0))|  + \int_{t_0}^t e^{(l-l^*) (s-t_0)}L \eta \chi(v(s)) \,ds \biggr)
    \end{equation}
    where $\chi$ is defined by \eqref{E:chi} in Theorem~\ref{Theorem: inverse mapping}.
    Moreover, there exists $L_2>0$ such that if the the switching interval $\mathcal{S}=\{t \ge 0: |y(t)|<\delta\}$ is non-empty, letting $t_1 = \inf\mathcal{S}$, it holds
    \begin{equation}\label{eq:bigthm:2}
      |v(t)-\hat{v}(t)|  \leq e^{L_2 (t-t_1)} |\hat{v}(t_1)-v(t_1)| , \quad t \in \mathcal{S},
    \end{equation}
    
    \item\label{it:main2} Suppose that there exist $\eta^*>0$ such that $|v_{1:2}(t)| \geq \eta^*, \ \forall t \geq 0.$
    Then for $\eta=\eta^*$, $0<\delta<\delta^*$, there exist $l^*>0$, $M>0$ and $k>0$ such that for any $l \geq l^*$ for any solution $\hat v$ of \eqref{Observer 1}  with initial conditions $\hat{v}(0) \in \R^3$ and $\hat z(0)= T_0(\hat v(0))$, it holds 
    \begin{equation}\label{eq:bigthm:3}
    |\hat{v}-v|\leq k|v(0)-\hat v (0)|e^{-(l-l^*) t}, \quad t \geq 0.
    \end{equation}

    \item\label{it:main3} For any time $t^*>0$, any compact $\mathcal{K}\subset \R^3$, and any $\varepsilon>0$, there exist $\delta>0$, $\eta>0$, $l > l^*$ sufficiently large such that for any solution $\hat v$ of \eqref{Observer 1}  with initial conditions $\hat{v}(0) \in \mathcal{K}$ and $\hat z(0)= T_0(\hat v(0))$, letting $t_2:=\inf \{ t \geq 0, \ y(t)=\delta \}=\sup\mathcal{S}$, it holds
    \begin{equation}\label{eq:bigthm:5}
    |\hat{v}(t)-v(t)| \leq \varepsilon , \quad \ t \in [t^*,t_2] \cup (t_2+t^*,\infty).
    \end{equation}
    \end{enumerate}
\end{thm}

Let us comment on this theorem. 
We described in Section \ref{sec:observability} a certain number of
observability losses in our model \eqref{EqP}-\eqref{EqP-y};
these occur for certain behaviors of the input and in some regions of
the state space.
Our approach has been to make assumptions on the input ---this
is Assumption~\ref{Strong Persistence}--- but to then proceed with crafting
an ad-hoc versatile observer ---this is \eqref{Observer 1}, based on the pseudo-inverse whose properties are
described by Theorem~\ref{Theorem: inverse mapping}--- capable of achieving tunable exponential
convergence when the state trajectory stays away from the identified
singularities, or do not approach them ``too often'', while also
ensuring acceptable performance for any other space trajectories.
Let us detail and comment these properties, stated formally above.
\begin{itemize}
\item Solutions of the observer are defined for all time for any initial condition, for any
  state trajectory of the system, this is Proposition~\ref{prop:observer_defined}.
\item It provides tunable exponential convergence when the state trajectory stays in some
  region $\{|v_{1:2}|>\eta^*>0\}$; this is point \eqref{it:main2} of the theorem.
  Note that the hybrid construction allows us to preserve exponential convergence through
  one possible passage through the non-obervability locus $\{v_0=0\}$.
  We explain in Remark~\ref{remark: Biological Interest.} to what extent these state
  trajectories are the ones of biological relevance. 
  One might chose to see this point as the main result. We refer the reader to Section~\ref{Section:Numerical Simulations} for an example of such a trajectory.
\item It provides rougher practical convergence when the state trajectory keeps passing
  in the ``bad regions'' as time grows; 
  this is point \eqref{it:main3}
  that gives an estimate that holds \emph{except} on a possible (arbitrarily small) time interval where $|v_0|$ is small. 
\end{itemize}
Point \eqref{it:main1} gives an explicit bound on the error
taking into account all contributions, without making any assumption on the state trajectory; this bound is possibly difficult to interpret but is necessary for completeness.
The other points may be deduced from the general estimations \eqref{eq:bigthm:1} and \eqref{eq:bigthm:2}. %

\section{Technical preliminaries} \label{Section Prelemenaries}

In this section, we prove the main properties of system \eqref{EqP} and the observability mapping $T$ introduced in~\eqref{eq:Introduction T}. We first start by proving the global existence of the solutions of \eqref{EqP} and the existence of attracting sets for this system.
\begin{proof}[Proof of Proposition~\ref{Solution defined Bounded}]
Recall that the function $\sigma$ is Lipschitz continuous with Lipschitz constant $\sigma'(0)$. Hence, letting $w = (1,r\cos(2\theta),r\sin(2\theta) )^\top$, we have for $v,\Tilde{v} \in \R^3 $:
\begin{equation}
    |\sigma(w^\top v) - \sigma(w^\top \tilde v)| \leq \sigma'(0) \sqrt{1+ r^2} \ |v-\Tilde{v}|.
\end{equation}
We therefore obtain,
\begin{equation}
        |f(v,t)-f(\Tilde{v},t)| \leq \bigl| \ -1 +  \ \sqrt{(J_0 ^2 +2 J_1 ^2)} \   \int_{[0,\infty)}\sigma'(0)(1+r^2)P(r) dr \ \bigr| \  |v-\hat v|,
\end{equation}
    which is then bounded uniformly on time as $P$ is compactly supported by assumption. Therefore solutions of \eqref{EqP} are well defined for every time.
    
    Let now $v$ be solution of \eqref{EqP}. Using the fact that $\sigma$ is bounded by 1, we have that: $$\dot{v}^{\top} v=\frac{1}{2}\frac{d |v|^2}{dt} \leq -|v|^2 + \Bigl( \sqrt{(J_0 ^2 + 2J_1^2)} \ \int_R(\sqrt{1+r^2}P(r)dr  + \sup_t ||I|| \Bigr) \ |v|,$$
    as the integral term is uniformly bounded.  
    Therefore, if $ \sup_t ||I|| < \infty$ there exist $R^*>0$ large enough such that for any $R\geq R^*$, if $|v| > R$ we have $\dot{v}^{\top}v < 0.$
\end{proof}

To study the observability mapping $T$, we utilise the $O(2)$ symmetry of the non-linear part of \eqref{EqP} and therefore study \eqref{EqP} from polar coordinate point of view. This allows us to naturally extend  the map $T_t$  into a diffeomorphism at each time $t$ and more easily construct a left inverse. 

We now put system \eqref{EqP} into polar coordinates. To this effect, we rewrite the system as
\begin{equation}
    \Dot{v}=-v + \Psi(v) + I,
\end{equation}
where $\Psi$ the non-linear integral part of $f$. For any $\phi\in [0,2\pi)$ we have
\begin{equation}\label{E:sym}
    \Psi\left(\ROT{\phi}v\right)=\ROT{\phi}\Psi(v),
\end{equation}
which is obtained simply through the change of variable $\theta'=\theta+\phi$ in the periodic integrals.
To take advantage of the above observation,
we introduce the following quantities, for $j,p\in\mathbb{N}$,
\begin{equation}\label{def-Gammapj}
    \Gamma_p^j(v_0,\rho) = \int_\Omega r^j \cos(2\theta)^j \sigma^{(p)}(v_0+r\rho\cos(2\theta)) P(r)\frac{d\theta dr}{\pi}, \qquad v_0\in\mathbb{R},\, \rho \geq 0.
\end{equation}
Here, $\sigma^{(p)}$ denotes the $p$-th derivative of the real-valued function $\sigma$, in particular $\sigma^{(0)}=\sigma$.
Letting $v_{1:2}=\rho e_{\phi}$ with $e_\phi=(\cos\phi,\sin\phi)^\top$,  via \eqref{E:sym} above we obtain

\begin{equation}
    \Psi(v) = \ROT{\phi} 
    \Psi  \begin{pmatrix}
        v_0 \\ |v_{1:2}| \\ 0
    \end{pmatrix} = \ROT{\phi} 
    \begin{pmatrix}
    J_0 \Gamma_0^0(v_0,\rho) \\
    J_1 \Gamma_0^1(v_0,\rho) \\
   0
   \end{pmatrix}.
\end{equation}

Here, we used the fact that 
\begin{equation}
    \int_\Omega r\sin(2\theta) \sigma(v_0+r\rho \cos(2\theta))\,\frac{d\theta}{\pi}P(r)dr = 0, \qquad \forall \rho\ge 0.
\end{equation}
As a consequence, the dynamics \eqref{EqP} reads
\begin{equation} \label{eq: Pre-Polaire}
    \begin{cases}
        \dot{v}_0=-v_0 + J_0 \Gamma_0^0(v_0,|v_{1:2}|) + I_0, \\[.5em]
        \dot{v}_{1:2}=-v_{1:2}+ J_1 \Gamma_0^1(v_0,|v_{1:2}|)\displaystyle\frac{v_{1:2}}{|v_{1:2}|} + I_{1:2}.
    \end{cases}
\end{equation}

The following lemma presents some basic facts about the non-linearities $\Gamma_{p}^j$ introduced in \eqref{def-Gammapj}.

\begin{Lemma} \label{Lemma: Gamma_p^i/ Inj Gamma_0^0}
The following assertions hold.
\begin{enumerate}[(i)]
    \item The quantities $\Gamma_{p}^{j}$ are well defined for all $j,p\in\mathbb{N}$ and belong to the class $C^{\infty}(\R\times[0,\infty),\R)$.
    \item  We have  the following relations
        \begin{equation}\label{eq:gamma-der}
            \partial_{v_0} \Gamma_{p}^{j}= \Gamma_{p+1}^{j}
            \quad\text{and}\quad \partial_{\rho} \Gamma_{p}^{j}= \Gamma_{p+1}^{j+1},
            \qquad\forall j,p\in \mathbb{N}.
        \end{equation}
    \item  For any $v_0\neq 0$, the function $\Gamma_0^0(v_0,\cdot)\in C^\infty([0,\infty),\R)$ is injective. %
    \item  For any $v_0\neq 0$ and $\rho >0$, we have $\B{1} \neq 0$ and of opposite sign w.r.t.~$v_0$. Moreover, $\BExpress{1} (v_0,0) = 0$ for any $v_0\in \R$.
    \item  For any $\rho\ge 0$, we have $\AExpress{0}(0,\rho) = 0$.
    \end{enumerate}
\end{Lemma}
\begin{proof}
We start with points \textit{(i)} and \textit{(ii)}.
For for $j, p \in\mathbb{N}$, we set % 
\begin{equation}
   \gamma_{p}^{j}(v_0,\rho,r,\theta):= \sigma^{(p)}\bigl(v_0 + r\rho \cos(2\theta) \bigr) r^j \cos(2\theta)^j P(r), \quad v_0 \in \R \ \rho \in [0,\infty), \ r \in [0,\infty), \ \theta \in \mathbb{S}^1,
\end{equation}
From our assumptions on $\sigma$ (see \eqref{Assumption Sigma}), 
for any fixed $\theta$ and almost every $r$, the function $\gamma_{p}^{j}(\cdot,r,\theta)$ is $C^{\infty}(\R\times[0,\infty),\R)$ %
and verifies
\begin{equation}
    \partial_{v_0}\gamma_{p}^{j}(\cdot,r,\theta)=\gamma_{p+1}^{j}(\cdot,r,\theta),
    \quad \text{and}\quad
    \partial_{\rho}\gamma_{p}^{j}(\cdot,r,\theta)=\gamma_{p+1}^{j+1}(\cdot,r,\theta).
\end{equation}
Furthermore, using the boundedness of $\sigma^{(p)}$ we see that
\begin{equation}
   | \sigma^{p}\bigl(v_0 + r\rho \cos(2\theta) \bigr) r^j \cos(2\theta)^j P(r) | \leq  \|\sigma^{(p)}\|_\infty r^j P(r).
\end{equation}
Using the fact that $P$ is a compactly supported probability density, the above allows to apply Lesbegue dominated convergence theorem, which yields the claim.

Point \textit{(iii)} is a direct consequence of point \textit{(iv)} as \textit{(ii)} implies $\partial_\rho\Gamma_{0}^{0}(v_0,\cdot) = \Gamma_1^1(v_0,\cdot)$, so we turn to a proof of \textit{(iv)}.

From Assumption~\ref{ass:sigma} on $\sigma$, by remarking that $\sigma'$ is even, strictly increasing on $(-\infty,0)$ and strictly decreasing on $(0,\infty)$ with $\sigma'(0)=\sup_{\mathbb{R}}\sigma'$, we get that: 
\begin{equation}
    \begin{cases}
     \sigma'(v_0+ r \rho U)-\sigma'(v_0- r \rho U)<0, \qquad \forall \rho>0, & \text{if} \ v_0>0 ,\\
     \sigma'(v_0+ r \rho U)-\sigma'(v_0- r \rho U)>0, \qquad \forall \rho>0, & \text{if} \ {v}_0<0, \\
     \sigma'(v_0+ r \rho U)-\sigma'(v_0- r \rho U)=0, \qquad  & \text{if} \  \rho {v}_0=0.
    \end{cases}
\end{equation}
We have then, using the fact that $r$ and $P(r)$ are positive, that
\begin{equation}
     \begin{cases}
     \B{1}<0, \qquad \forall \rho>0, & \text{if} \ v_0>0 ,\\
     \B{1}>0, \qquad \forall \rho>0, & \text{if} \ {v}_0<0\\
     \B{1}=0,  \qquad  & \text{if} \  \rho {v}_0=0.
    \end{cases}
\end{equation}
This completes the proof of \textit{(iv)}.

Finally, point \textit{(v)} follows at once from direct computations and the parity of $\sigma$, which yields
\begin{equation}
    \AExpress{0}(0,\rho) = \int_\Omega \sigma(r\rho\cos(2\theta))P(r)\frac{d\theta dr}{\pi} = 0.
\end{equation}
\end{proof}

\begin{remark}
    As already mentioned, our results are true without the compact support assumption on $P$. The results can be obtained with the assumptions that the maximal moment of $P$, which we denotes $m$, is greater then 3. Indeed, the above result still holds, supposing that the distribution $P$ has a moment $m$ greater or equal then 2. Indeed by replacing point \textit{(i)} 
    by the fact that $\Gamma_{p}^{j}$ is well defined for all $j \leq m$ and has regularity $C^{(m-j)}(\R\times[0,\infty),\R)$.
\end{remark}

We are now in a position to prove Propositions~\ref{Observability rho} and \ref{Proposition : Passage}.

\begin{proof}[Proof of Proposition~\ref{Observability rho}] \label{Proof: Observability rho}
Let $v$ and $\tilde{v}$ be two solutions of~\eqref{EqP}.
Let us prove \textit{(i)}. From Lemma~\ref{Lemma: Gamma_p^i/ Inj Gamma_0^0}, $\Gamma_0^0 \equiv 0$. Using Equation~\eqref{eq:Dynamique Polar Coordinate.}, the set $\mathcal{Z}_0=\{ v\in \R^3, \ v_0=0 \}$ is stable by the dynamic $f$ as $I_0(t)=0$ for all times $t\geq 0$. Therefore for solution initialised in the set $\mathcal{Z}_0$ the dynamic can be reduced to the close form
    \begin{equation}
        \dot{v}_{1:2}=f_{1:2}((v_0,v_1,v_2),I(t)), \qquad v_0=0,
    \end{equation}
    which admits a unique solutions for any initial condition $v_{1:2}(0) \in \R^2$ as $f$ is globally Lipschitz continuous from \ref{Solution defined Bounded}. There exist then solutions of \eqref{EqP} such that 
    \begin{equation}
        v_0(t)=\title{v}_0(t), \ \forall t \geq 0, \ \text{and} \ v(t) \neq \Tilde{v}(t), \ \forall \ t \geq 0.  
    \end{equation}
    
Let us now prove point \textit{ii}.
By assumption, it holds 
\begin{equation}
        v_0(t)=\tilde{v}_0(t) \quad \textrm{for}\quad t\in [0,t^*].
\end{equation}
In particular, $\dot v_0 \equiv \dot{\tilde{v}}_0$ on $(0,t^*)$. Hence, using the first equation of \eqref{eq: Pre-Polaire} and the fact that $J_0\neq 0$, we obtain that
\begin{equation}
    \AExpress{0}(v_0(t),\rho(t))=\AExpress{0}(v_0(t),\Tilde{\rho}(t)) \quad \textrm{for}\quad t\in [0,t^*].
\end{equation}

For $t\in \mathcal I = \{t\in [0,t^*]:\: v_0(t)\neq 0\}$, point \textit{(iii)} of Lemma~\ref{Lemma: Gamma_p^i/ Inj Gamma_0^0} guarantees that $\Gamma_0^0(v_0(t),\cdot)$ is injective, and thus that
\begin{equation}
    \rho(t)=\Tilde{\rho}(t), \quad \text{for} \quad t \in \mathcal{I}.
\end{equation}
This proves that $|v_{1:2}| = |\tilde{v}_{1:2}|$ for $t\in \mathcal{I}$. Differentiating this equality, which is possible since $\mathcal{I}$ is a relatively open subset of $[0,t^*]$, and using again \eqref{eq: Pre-Polaire} allows to conclude the first step of the proof.
\end{proof}

\begin{proof}[Proof of Proposition~\ref{Proposition : Passage}]

From \eqref{def-Gammapj} we have $\Gamma_0^0(v_0,0)=\sigma(v_0)$. From Lemma~\ref{Lemma: Gamma_p^i/ Inj Gamma_0^0} one has for all $\rho \geq 0$ that $\Gamma_1^1(v_0,\rho) \leq  0$ (respectively  $\Gamma_1^1(v_0,\rho) \geq  0$) for all $v_0 \geq 0$ (respectively for all $v_0 \leq 0$). From the fact that $\partial_{\rho} \Gamma_0^0=\Gamma_1^1$ and by continuity, this yields  
\begin{equation}
    |\Gamma_0^0(v_0,\rho)| \leq |\sigma(v_0)|, \quad \forall (v_0,\rho) \in \R\times [0,\infty).
\end{equation}
From  Assumption~\ref{ass:sigma} on $\sigma$, one has $|\sigma(v_0)| \leq \sigma'(0)|v_0|$. Substituting this in \eqref{eq: Pre-Polaire}, together with the lower bound on $I_0(t)$ yields 
\begin{equation} \label{eq: Estimation v_0}
    \dot v_0\geq -|v_0|- |J_0| \sigma'(0) |v_0|+c,
\end{equation} 
whence, for any solution $v(\cdot)$ of \eqref{EqP} with an input satisfying the assumptions of the proposition,
\begin{equation}
    \label{eq:ineq_f0_pf}
\dot v_0(t)>0 
\quad \text{for all $t$ such that}\quad
|v_0(t)|<\delta^*\,.
\end{equation}
Similarly, the time estimation is obtained from applying Grönwall's lemma to \eqref{eq: Estimation v_0}.

Now, consider a solution $v(\cdot)$ and remember that $\delta>0$ is smaller than $\delta^*$. 
 If $v_0(0)<-\delta$, either $v_0(0)$ remains smaller than $-\delta$ for all time, and we are in the first case of the proposition, or there is a smaller time $t_1>0$ such that $v_0(t_1)=-\delta$, and in that case, \eqref{eq:ineq_f0_pf} implies that we must be in the second case of the proposition because $v_0(t)$ increases as long as it is smaller than $\delta^*$, hence it passes through $\delta$ at some finite time $t_2$ and never becomes smaller than $\delta$ in the future.
If $-\delta\leq v_0(0)<\delta$, by the same argument, there is a time $t_2$ so that $v_0(t)$ has the behavior describes as the third case of the Proposition, and if $v_0(0)\geq\delta$, which is the last possibility, the same arguments tell us that $v_0(t)$ is larger than $\delta$ for all positive $t$.
\end{proof}
\section{The observability mapping} \label{Section: Observability Mapping}

In this section we focus on the observability of system \eqref{EqP}. For any solutions of \eqref{EqP} with measurement $h(v)=v_0$, we want to show the injectivity between the function $v_0(\cdot)$ and its time derivatives, up to certain order, with the solution $v(\cdot)$.

To do so, we look at the properties of the mapping $T$ defined in~\eqref{eq:Introduction T}, and in particular we aim at constructing the pseudo-inverse of Theorem~\ref{Theorem: inverse mapping}. In order to do so, we start by extending $T_t$ to a diffeomorphism, for each time $t \ge 0$.

\subsection{Observability mapping extension and differential observability} 
\label{Section 3.2}

From the symmetry of the system, we want to study $T$ from a polar coordinate point of view. 
Letting $X= (v_0, \rho, e_\phi)^\top \in \R \times (0,\infty) \times \{ \zeta \in \R^2 : |\zeta|^2=1 \}$,
(recall $v_{1:2}=\rho e_{\phi}$ with $e_\phi=(\cos\phi,\sin\phi)^\top$), we obtain from \eqref{eq: Pre-Polaire} the following dynamics in polar coordinates
\begin{equation} \label{eq:Dynamique Polar Coordinate.}
    \dot{X}= \begin{bmatrix}
        -v_0 + J_0\A{0} + I_0 \\
        -\rho + J_1 \B{0} + I_{1:2}^{\top} e_{\phi} \\
        \frac{1}{\rho}( -( I_{1:2}^{\top}e_{\phi} )  e_{\phi} +    I_{1:2})
        \end{bmatrix}.
\end{equation}

We can naturally extend \eqref{eq:Dynamique Polar Coordinate.} by changing the variable $e_\phi \in S^1$, where $S^1$ is the unit circle of $\R^2$ to $\zeta \in \R^2$. We obtain then a new dynamic  $\dot X = F(X,t)$ on $\R \times (0,+\infty) \times \R^2$ as follow

\begin{equation} \label{eq: F extention}
    F(X,t)= \begin{bmatrix}
        -v_0 + J_0\A{0} + I_0 \\
        -\rho + J_1 \B{0} + I_{1:2}^{\top} \zeta \\
        \frac{1}{\rho}( -( I_{1:2}^{\top}\zeta )  \zeta   +    I_{1:2}) \end{bmatrix},
    \qquad\text{where}\quad
    X = \begin{pmatrix}
        v_0 \\ \rho \\ \zeta
    \end{pmatrix}\in\R \times (0,+\infty) \times \R^2 .
\end{equation}

We define
\begin{equation}
    \mathcal{X}:=\mathbb{R}^*\times (0,\infty) \times  \mathbb{R}^2.
\end{equation}

Let us assume that $I\in C^2([0,+\infty),\R^3)$. With $\D$ the operator defined in \eqref{eq: operator differentiel D}, and omitting the arguments of $\Gamma_p^j$ (defined in~\eqref{def-Gammapj}),
we define 
$S_{t}:\mathcal{X}\to\mathbb{R}^4$ by
\begin{equation}
    \label{eq:St-expression}
    \begin{split}
    [S_t(X)]_0=& \ v_0 \\ 
    [S_t(X)]_1=&-v_0 + J_0 \Gamma_0^0   + I_0\\
    [S_t(X)]_2=&- F_0 +  J_0\Gamma_1^0  F_0 
    + J_0\Gamma_1^1 F_1  + \dot I _0 \\ 
    [S_t(X)]_3=&
    -\D F_0 +J_0 \Gamma_1^0 \D F_0 + 
    J_0\Gamma_2^0  (F_0)^2 \\& + 2 J_0\Gamma_2^1 F_1  F_0 + 
    J_0\Gamma_1^1 \D F_1  + J_0\Gamma_2^2 ( F_1 )^2 + \ddot I_0.
    \end{split}
\end{equation}
From direct computations, one can see that 
for any $t\ge 0$, any $X\in \R^* \times (0,+\infty) \times \{\zeta\in\mathbb{R}^2:\:|\zeta|=1\}\to \R^4$ it holds
\begin{equation}
       S_t\left(X \right) := T_t(X_0,X_1X_{2:3}). 
\end{equation}

Recall that by Assumption~\ref{Strong Persistence} we can fix an arbitrary large $R>0$ such that $B_{\R^3}(0,R)$ is an invariant attracting set for \eqref{EqP}. 
For any $\delta>0$, $\eta>0$,  we set
    \begin{equation}
        \mathcal{X}_{\delta,\eta} = \{ (v_0,\rho,\zeta)\in \mathcal X :\: \delta\le |v_0| \le R,\, \eta\le \rho\le R,\, |\zeta| \le \ R\}.
    \end{equation}
We then have the following.

\begin{thm} \label{Theorem 1 S inj imm}
    Assume $I\in C^{2}([0,+\infty),\R^3)$. Let 
    $t\ge 0$ be such that $I_{1:2}(t)\wedge \dot I_{1:2}(t)\neq 0$. 
    The map $S_{t}:\mathcal{X}\to\mathbb{R}^4$ is a $C^\infty$  
    diffeomorphism onto its image.
    In particular, under Assumption~\ref{Strong Persistence} the map $S_t$ is invertible for all times $t\ge 0$ and 
    the inverse $S_t^{-1}: S_t(\mathcal{X}) \mapsto \mathcal{X}$ satisfies
    \begin{equation}
        \sup_t \left\{\|\partial_z S_t^{-1}(z) \| : z\in S_t(\mathcal{X}_{\delta,\eta})\right\} < \infty.
    \end{equation}
\end{thm}

\begin{proof}\label{Preuve Diffeo S}
    Let $t\ge 0$ be such that $I_{1:2}(t)\wedge \dot I_{1:2}(t)\neq 0$.
    The map $S_t$ is defined as composition of the functions $\Gamma_{p}^{j}$ and $\D^{i}F$, $ 0 \leq i \leq 2$. From equation~\eqref{eq: F extention}, $F$ is defined for all $X \in \mathcal{X}$.
    As such, it is $C^\infty$ over $\mathcal{X}$.
   In the following we denote elements of $\mathcal{X}$ by $X=(v_0, \rho, \zeta)$, where $v_0\in \R^*$, $\rho\in (0,+\infty)$, and $\zeta\in \R^2$.
    
\textit{Step 1. Injectivity of $S_t$.}
Let $X,\Tilde{X}\in \mathcal{X}$ be such that $S_t(X)=S_t(\tilde{X})$. We will compare term by term the expression of $S_t$ given in \eqref{eq:St-expression}. Observe that since $v_0\neq 0$ by definition of $\mathcal X$, Lemma~\ref{Lemma: Gamma_p^i/ Inj Gamma_0^0} implies injectivity of $\Gamma_0^0(v_0,\cdot)$ over $(0,\infty)$ and $\Gamma_1^1(v_0,\cdot)\neq 0$.

By the expression of the first two components of $S_t$, we immediately have $$v_0=\Tilde{v}_0
\quad\text{and}\quad
\rho=\Tilde{\rho}.$$
Using the expression \eqref{eq: F extention} of $F$ one verifies that all terms in the expression of $[S_t]_2$ depend only on $\rho$ and $v_0$,  except $\B{1}F_1(X,t)$. The fact that $\B{1}\neq 0$, and the expression of $F_1(X,t)$ then implies that
\begin{equation}\label{eq:boh}
\IT\zeta=\IT \Tilde{\zeta}.
\end{equation}

We now turn to the expression of $[S_t]_3$. By using again the expression \eqref{eq: F extention} of $F$, eliminating all terms that depend only on $\rho$ and $v_0$, and using the fact that $\B{1}\neq 0$, we obtain $\D F_1(X,t)=\D F_1(\tilde X,t)$, which is the term that contains the last unknown term. 
Developing this expression and using \eqref{eq:boh}, we obtain
\begin{equation}\label{eq:boh2} 
\dot{I}_{1:2}^{\top}\zeta=\dot{I}_{1:2}^{\top} \Tilde{\zeta}.
\end{equation}

Putting together \eqref{eq:boh} and \eqref{eq:boh2}, we finally obtain 
\[
\begin{bmatrix}
    \IT \\ \dot{I}_{1:2}^{\top}
\end{bmatrix} \begin{bmatrix}
    \zeta - \Tilde{\zeta}
\end{bmatrix}=0.
\]
This implies that $\zeta=\tilde \zeta$ since $I_{1:2}(t)\wedge \dot I_{1:2}(t)\neq 0$.
This ends the proof of the injectivity.

\textit{Step 2. $S_t$ is a diffeomorphism onto its image.}
As we have seen that $S_t$ is differentiable and injective, by the rank theorem we just need to show that its differential $\partial_X S_t$ is of maximal rank on $\mathcal{X}$. 
For any $X = (v_0,\rho,\zeta) \in \mathcal{X}$, direct computations yield
\begin{equation}
    \label{eq:DXSt}
    \partial_X S_t(X)=\begin{bmatrix}
        1 & 0 & 0_{1\times2}\\
        \star &  J_0\B{1} & 0_{1\times2} \\
        \star & \star & J_0\B{1} G(v_0,\rho,\zeta) \\
    \end{bmatrix}.
\end{equation}
Here, we let $0_{1\times 2} = (0,0)\in \R^{1\times 2}$, and
\begin{equation}
\label{eq:GI}
    G(v_0,\rho,\zeta):= \begin{bmatrix}
    \IT \\
    \dot{I}_{1:2}^{\top}
\end{bmatrix}
 + \tilde U(v_0,\rho,\zeta)
 \begin{bmatrix}
     0_{1\times 2} \\ \IT
 \end{bmatrix}
\end{equation}
for some function $\tilde U$. We omit the expressions noted as $(\star)$ in \eqref{eq:DXSt} as they do not intervene at this level. 

As we are under \eqref{Assumption Sigma} and $X \in \mathcal{X}$, we have $\B{1} \neq 0$ by Lemma ~\ref{Lemma: Gamma_p^i/ Inj Gamma_0^0}. Therefore as the Jacobian matrix of $S$ is triangular by block, to prove that it is invertible we are left to prove the invertibility of $G(v_0,\rho,\zeta)$.
This follows by assumption, since
\begin{equation}
    \label{eq:det-U}
    \det G(v_0,\rho,\zeta) = \dot{I}_{1:2}\wedge I_{1:2} \neq 0
\end{equation}
Therefore $S_t$ is an immersion in $\mathcal{X}$.

\textit{Step 3. Estimates on the inverse with respect to time.}
The inverse diffeomorphism $S_t^{-1}$ defined on $S_t(\mathcal{X})$ is such that
$\partial_z S_t^{-1}(z) = (\partial_X S_t(S_t^{-1}(z)))^{-1}$. Since the dependence on $t$ of $\partial_X S_t$ is only due to the presence of $I_{1:2}$, $\dot I_{1:2}$, $I_0$ and $\dot I_0$, and recalling that $J_0\BExpress{1}$ is non null, we rewrite it as follows (we omit the dependencies on $(v_0,\rho,\zeta)$):
\begin{equation}
    \label{eq:DXSt2}
    \partial_X S_t(X)=\begin{bmatrix}
        1 & 0 & 0_{1\times2}\\
        \alpha &  J_0\BExpress{1} & 0_{1\times2} \\
         J_0\BExpress{1}\beta(I) &  J_0\BExpress{1} \gamma(I) &  J_0\BExpress{1} G(I) \\
    \end{bmatrix} 
    = 
    \underbrace{
    \begin{bmatrix}
        1 & 0 & 0_{1\times2}\\
        \alpha & J_0\BExpress{1} & 0_{1\times2} \\
        0 & 0 & J_0\BExpress{1}  \operatorname{Id}_{2} \\
    \end{bmatrix} 
    }_{=:M}
    \underbrace{
    \begin{bmatrix}
        1 & 0 & 0_{1\times2}\\
        0 & 1 & 0_{1\times2} \\
        \beta(I) & \gamma(I) & G(I)  \\
    \end{bmatrix}
    }_{=:N(I)},
\end{equation}
for some function $\alpha$ independent of the input $I$ and functions $\beta,\gamma$ depending on $I$ and its first derivative. Here, we denoted by $G(I)$ the matrix in \eqref{eq:GI}, in order to stress its dependency on $I$.
This expression yields that 
\[
|\partial_z S_t^{-1}| \le |N(I)^{-1}| |M^{-1}| 
\]
Observe that $M$ is continuous on $\mathcal X$ (albeit singular if $\rho v_0 = 0$), and thus $\|M^{-1}\|_{S_t(\mathcal{X}_{\delta,\eta})}\le c$ for some constant $c>0$ depending only on $\delta$ and $\eta$. Moreover,
\[
N(I)^{-1}
=
\begin{bmatrix}
        1 & 0 & 0_{1\times2}\\
        0 & 1 & 0_{1\times2} \\
        -G^{-1}(I)\beta(I) & -G^{-1}(I)\gamma(I) & G^{-1}(I)   \\
    \end{bmatrix}.
\]
The lower bound on the determinant of $G$ given by Assumption~\ref{Strong Persistence} together with the fact that $I$ and $\dot I$ are uniformly bounded with respect to time allows then to conclude.
\end{proof}

We now prove Theorem~\ref{prop:T-inj} as a corollary of Theorem~\ref{Theorem 1 S inj imm}. Recall the notations $\mathcal{Z}_0=\{ v \in \R^3, v_0=0\}$ and  $\mathcal{Z}_{1:2}=\{v\in \R^3, v_1=v_2=0\}$.

\begin{proof}[Proof of Theorem~\ref{prop:T-inj}]
Let $v \in \R^3\setminus \mathcal{Z}_0$. %, i.e., such that $v_0 \neq 0$. 
Suppose at first $\rho=|v_{1:2}| >0 $. We have by construction that $T_t(v)=S_t(X)$ where $X=(v_0,\rho,\frac{v_{1:2}}{\rho})\in \mathcal{M}:=\R^* \times (0,+\infty) \times \{|\zeta|=1\}$. Since Theorem~\ref{Theorem 1 S inj imm} guarantees that $S_t$ is a diffeomorphism of $\mathcal{X}$ on $S_t(\mathcal X)$, it is therefore an injective immersion when restricted to the sub-manifold of dimension 3 $\mathcal{M} \subset \mathcal{X}$. 
It follows that  $T_t$ is also an injective immersion for $v \in \R^3 \setminus (\mathcal{Z}_0\cap \mathcal{Z}_1)$.

If $|v_{1:2}|=0$, let $\tilde v$ be such that $T_t(v)=T_t(\tilde{v})$. By computing the first two elements of the expression of $T_t$, one deduces that $\AExpress{0}(v_0,|v_{1:2}|)=\AExpress{0}(v_0,|\tilde{v}_{1:2}|)$. We deduce from Lemma~\ref{Lemma: Gamma_p^i/ Inj Gamma_0^0} that this function is injective, hence that $|v_{1:2}|=0$. Thus, $v=\Tilde{v}$ and $T_t$ is indeed injective on $\R^3 \setminus \mathcal{Z}_0$.  
\end{proof}

We can now prove the observability of the system as stated in Proposition~\ref{Observability Theorem}.

\begin{proof}[Proof of Proposition~\ref{Observability Theorem}]
We start by proving \textit{(ii)} $\Rightarrow$ \textit{(i)}.
To this aim, let $v$ and $\Tilde{v}$ be two solutions of \eqref{EqP} such that $h(v(t))=h(\tilde v(t))$ for all $t\in [0,t^*]$. By the definition of $T_t$ in \eqref{eq:Introduction T} we have that $T_t(v(t))=T_t(\Tilde{v}(t))$ for all $t\in [0,t^*]$. 
By assumption and by continuity of $t\mapsto I_{1:2}(t)\wedge \dot I_{1:2}(t)$ there exists an open interval {$\mathcal{I}\subset [0,t^*]$} such that $I_{1:2}(t)\wedge \dot I_{1:2}(t)\neq 0$ for all $t\in \mathcal{I}$.  
By the fact that $I_0(t)\ge c$ and Proposition~\ref{Proposition : Passage}, it follows that there exists $t_0\in \mathcal{I}$ such that $v_0(t_0)\neq 0$. 
Thus, by Theorem~\ref{prop:T-inj} it follows that $T_{t_0}$ is injective. 
This yields $v(t_0)=\tilde v(t_0)$, which proves the claim by uniqueness of solutions of \eqref{EqP}.

We now turn to an argument \textit{(i)} $\Rightarrow$ \textit{(ii)}.
In this case, due to the assumption that $
I_0(t)>c $ for all $t\ge 0$, we get from Proposition~\ref{Proposition : Passage} that there exist at most a single time $t_c \geq 0$ where $v(t_c)=0$. From Proposition~\ref{Observability rho} we have that $|v_{1:2}|=|\Tilde{v}_{1:2}|$ and $\I^{\top} v= \I^{\top} \Tilde{v}$ a.e. 
If $v$ and $\tilde v$ are two distinct solutions of \eqref{EqP} such as $v_0=\tilde v_0$ and $I_{1:2}(0)^{\top} \bigl( v_{1:2}(0)-\tilde{v}_{1:2}(0))=0$ then it is direct to see that
\begin{equation}
    h(v(t))=h(\tilde{v}(t)), \ \forall t\ge 0  \ \And \ v(t)\neq\Tilde{v}(t), \ \forall t \ge 0,
\end{equation}
which ends the proof.
\end{proof}

\subsection{The pseudo-inverse}  \label{Section 3.3}

The goal of this section is to construct a pseudo-inverse of the observability mapping $T_t$, that is, to construct a  Lipschitz map $\mathfrak{T}_t:\R^4\to \R^3$ such that $\mathfrak{T}_t\circ T_t$ is the identity on (most of) $\R^3$. This map $\mathfrak{T}_t$ is a fundamental building block of our observer.
Under the assumptions of Theorem~\ref{Theorem 1 S inj imm}, it is not difficult to provide an inverse of the map $T_t$ since, letting $\Phi(X):=(X_0,X_1 X_2,X_1 X_3)$ for $X\in\R^4$, it holds
\begin{equation}
    \label{eq:t-true-inverse}
    \Phi \circ S^{-1}_t \circ T_t = \operatorname{Id} \quad \text{on } \mathbb{R}^3\setminus\mathcal{Z}.
\end{equation}
The above cannot be directly used to obtain a Lipschitz inverse since, as can be seen in the proof of Theorem~\ref{Theorem 1 S inj imm}, the Lipschitz constant of $S_t^{-1}$ blows up when $\rho$ or $v_0$ are small.
However, $S_t^{-1}$ is Lipschitz on $S_t(\mathcal{X}_{\delta,\eta})$, which suggests defining an inverse via \eqref{eq:t-true-inverse} by inserting a projection on $S_t(\mathcal{X}_{\delta,\eta})$ in the definition. However, due to the complicated structure of $S_t(\mathcal{X}_{\delta,\eta})$, this is unpractical, and thus we resort to the following definition for the pseudo-inverse $\mathfrak{T}_t$:
\begin{equation}
    \label{eq:t-inverse}
    \mathfrak{T}_t = \Phi \circ \Sat \circ S_t^{-1} \circ \Pi_t.
\end{equation}
Here, $\Pi_t$ is a projection depending on the output $y(t)=v_0(t)$ and $\Sat$ is an appropriate cut-off function, bounding the non-linear part in the observer and thus insuring that trajectories do not explode in finite time.

Namely, observe that
a necessary condition for $z$ to belong to $S_t(\mathcal{X}_{\delta,\eta})$ is that 

\begin{equation}
    \label{eq:necessary}
 |z_0|>\delta, 
 \qquad\text{and}\qquad
\begin{cases}
    J_0\AExpress{0}(z_0,R) \leq  z_1 + z_0 - I_0 \leq J_0\AExpress{0}(z_0,\eta)& \text{ if $z_0>0$,}
    \\
   J_0\AExpress{0}(z_0,\eta)\leq  z_1 + z_0 - I_0 \leq  J_0\AExpress{0}(z_0,R)& \text{ if $z_0<0$.}
\end{cases}
\end{equation} 
Hence, $\Pi_t:\mathbb{R}^4\to \mathbb{R}^4$ is defined as the projection on the above set. That is, 

\begin{eqnarray}
    \Pi_t(z)_0 &=& 
    \begin{cases}
        \min \{ R,  \max\{\delta, z_0\} \} & y(t) \ge 0, \\
        \max \{ -R, \min \{ -\delta,z_0\} \} & y(t) < 0, 
    \end{cases}\\
    \Pi_t(z)_1 &=&
    \begin{cases}
        g_{\eta,R}(\Pi_t(z)_0,z_1), \qquad y(t) \ge  0,  \\
        g_{R,\eta}(\Pi_t(z)_0,z_1), \qquad y(t) < 0,
    \end{cases}\\
    \Pi_t(z)_2 &=& z_2, \\
    \Pi_t(z)_3 &=& z_3.
\end{eqnarray}
Here, we let the auxiliary function $g$ to be defined by
\begin{equation}
    g_{\rho_1,\rho_2}(z_0,z_1) = 
    \begin{cases}
        J_0\AExpress{0}(z_0,\rho_1) -z_0+ I_0& \qquad \text{ if } z_1 > J_0\AExpress{0}(z_0,\rho_1) -z_0+ I_0,\\
        J_0\AExpress{0}(z_0,\rho_2) -z_0 + I_0& \qquad \text{ if } z_1 < J_0\AExpress{0}(z_0,\rho_2) -z_0 + I_0,\\
        z_1 & \qquad \text{otherwise.}
    \end{cases}
\end{equation}
It is clear that $\Pi_t$ is a Lipschitz function with Lipschitz constant uniformly bounded w.r.t.~time. 

The image of the projection $\Pi_t$ fails to be in $S_t(\mathcal{X}_{\delta,\eta})$ due to not respecting the $R$ upper bound in its last two components. 
Although we can still inverse $S_t^{-1}$ on the image of $\Pi_t$, 
in order to avoid solutions of the observer \eqref{Observer 1} from blowing up in finite time, we add an additional cut-off. 
Let $p:\R\to [0,1]$ be a smooth function such that $p(x)=1$ if $|x|\leq R-1$, $p(x)=0$  if $|x|\geq R$ (recall $R$ has been assumed to be arbitrarily large, and in particular larger than 3). Then
\begin{equation}\label{E:sat}
    \Sat(X) = 
    (X_0,X_1, p(|X_{2:3}|)X_{2:3}).
\end{equation}

We are now in a position to prove Theorem~\ref{Theorem: inverse mapping}.

\begin{proof}[Proof of Theorem~\ref{Theorem: inverse mapping}]
For each $t \geq 0$, the function $ \mathfrak{T}_t: \mathbb{R}^4 \to \mathbb{R}^3$ introduced in Equation~\eqref{eq:t-inverse} is well defined. Since $\Pi_t(\R^4) = S_t(\{X \in \R^4, \ |v_0|>\delta,\eta<\rho<R, \zeta \in \R^2 \})$ where the inverse $S_t^{-1}$ is well defined and $C^{\infty}$ due to Theorem~\ref{Theorem 1 S inj imm}, it is a composition of continuous functions. 
The image of the map $\mathfrak{T}_t$ is
\begin{equation}
   \mathfrak{T}_t(\R^4)= \Phi \circ \Sat \circ S_t^{-1} \circ \Pi_t (\R^4) =\Phi ( \SetXdelta) \subset \{ \delta \leq |v_0|\leq R \} \times B_{\R^2}(0,R^2).
\end{equation}
Moreover $\mathfrak{T}_t$ is globally Lipschitz continuous. Indeed,
it is a composition of locally Lipschitz functions. The map $S_t^{-1}\circ \Pi_t$ is not globally Lipschitz continous however (due to the last two components).
By construction of the cut-off function, we have
\begin{equation} \label{eq:norm eq}
 \lVert \partial_z \Sat\circ S_t^{-1} \rVert_{ \Pi_t(\R^4)} 
 =
 \lVert (\partial_X \Sat) (\partial_z S_t^{-1}) \rVert_{S_t(\SetXdelta)} 
 \leq 
 \lVert \partial_X \Sat \rVert \
 \lVert \partial_z S_t^{-1} \rVert_{S_t(\SetXdelta)}.
\end{equation}
Since $\Phi$ is smooth and thus locally Lipschitz continuous, and  $\Sat  \circ   S_t^{-1} \circ \Pi_t(\R^4) \subset \SetXdelta$, we have by construction
that $\mathfrak{T}_t$ is a globally Lipschitz continuous map for every time $t$. We also have that  $\Pi_t$ being globally Lpischitz continous uniformly with respect to time. Hence, as a consequence of \eqref{eq:norm eq} and Theorem~\ref{Theorem 1 S inj imm}, Lipschitz constant for $\mathfrak{T}_t$ exhibits uniform boundedness with respect to time.

    We now provide an argument for \eqref{eq:estimate-t-inverse}. Let $z=T_t(v)$ for some $v\in \mathbb{R}^3$ such that $v_0$ has the same sign as the input $y(t)$ and $|v_0|>\delta$. If $\eta<|v_{1:2}|\le R$, since $\Pi_t(z)=z$,  $S_t^{-1}(z)=\left(v_0,|v_{1:2}|,\frac{v_{1:2}}{|v_{1:2}|} \right)$. It is then immediate to check that \eqref{eq:t-inverse} reduces to \eqref{eq:t-true-inverse}, that is $\mathfrak{T}_t(z) = v$.
    On the other hand, if $0<|v_{1:2}|\le \eta$, direct computations show that
    \begin{equation}
        \mathfrak{T}_t(z)-v = \left(0, (\eta -|v_{1:2}|) \frac{v_{1:2}}{|v_{1:2}|} \right) \implies |\mathfrak{T}_t(z)-v|\le \eta.
    \end{equation}
    By continuity, the above inequality extends to $|v_{1:2}|=0$.
\end{proof}

\section{Observer convergence} 
\label{Section 4}

In this section we prove Theorem~\ref{Theorem Observer 1} concerning the convergence properties of Observer~\eqref{Observer 1}.

We need the following Lemma, insuring that the non-linear part of the observer is globally Lipschitz.

\begin{Lemma} \label{Lemma:Non Linear Part Lipschitz}
Under Assumption~\ref{Strong Persistence}, for every $t \geq 0$ the map $(\D^4h )_t\circ \mathfrak{T}_{t}:\mathbb{R}^4\to \mathbb{R}$
is Lipschitz continuous uniformly with respect to time. Here, we let $(\D^4 h)_t:=(\D^4 h)(\cdot,t)$. 
\end{Lemma}
\begin{proof}
Observe that $(\D^4 h)_t$ is well-defined for $t\ge 0$, since $I\in C^3([0,+\infty),\R^3)$ by Assumption~\ref{Strong Persistence}.
Moreover, $(\D^4 h)_t$ is locally Lipschitz, uniformly with respect to time. Indeed, being the composition of smooth functions, $(\D^4 h)_t$ is locally Lipschitz (actually smooth) for any $t\ge 0$. The uniformity with respect to time follows by observing that the time dependent terms in the Jacobian matrix $\partial_v (\D^4 h)_t$ are linear combinations of $I_i^{(j)}$, $i\in\{0,1,2\}$ and $j\in \{0,1,2,3\}$, which  are bounded in time by Assumption~\ref{Strong Persistence}.

Observe that by Theorem~\ref{Theorem: inverse mapping} the map $\mathfrak{T}_t$ is globally Lipschitz, uniformly with respect to time, and has bounded image.
Therefore, $\D^4h\circ \mathfrak{T}_t$ is also globally Lipschitz continuous uniformly with respect to time.
\end{proof}

We now prove the well definiteness of the observer~\eqref{Observer 1}.

\begin{proof}[Proof of Proposition~\ref{prop:observer_defined}]
    Thanks to Lemma~\ref{Lemma:Non Linear Part Lipschitz} the first equation of \eqref{Observer 1} has globally defined solutions. The same is true for the second equation by Proposition~\ref{Solution defined Bounded}. 
    By Proposition~\ref{Proposition : Passage} we have that there exists at most two times $t_2>t_1\ge 0$ such that $y(t_1) = -\delta$ and $y(t_2)=\delta$. Hence, $\hat v$ is piece-wise continuous with at most two jump discontinuities.
    However, $\hat v$ is continuous at $t_1$. Indeed, in this case the initial condition after the switch of dynamics is $\hat v(t_1^+) = \mathfrak{T}_t(\hat z(t_1^-))$. This completes the proof.
\end{proof}

\begin{remark}
    Let $t_2$ be the time such that $y(t_2)=\delta$. In this case, when switching dynamics in the observer \eqref{Observer 1} we have $\hat z(t_2^+) = T_t(\hat v(t_2^-))$. This can induce a discontinuity in $\hat v$ since 
    \begin{equation}
        \hat v(t_2^+) = \mathfrak T_t (\hat z(t_2^+)) = \mathfrak T_t \circ  T_t(\hat v(t_2^-)).
    \end{equation}
    Indeed, it could happen that $|\hat v_{1:2}(t_2^-)| <\eta$, and thus Theorem~\ref{Theorem: inverse mapping} cannot be used to guarantee that $\mathfrak T_t\circ T_t(\hat v(t_2^-)) = \hat v(t_2^-)$.
\end{remark}

We are now in position to prove Theorem~\ref{Theorem Observer 1}. 
This proof follows the conventional proof of the convergence of high gain design (see \cite{khalil2014high}, \cite{gauthier_kupka_2001} or \cite{besanccon2007nonlinear} for example) with the necessary modifications to address the singularities associated with the set $\mathcal{Z}$.

\begin{proof}[Proof of Theorem~\ref{Theorem Observer 1}]
As $0<\delta<\delta^*$ and $0<\eta<R$, we proved in Proposition~\ref{prop:observer_defined} that, under these assumptions, solutions of \eqref{Observer 1} are well defined globally.

We start by proving Point~\eqref{it:main1}.
We suppose at first that the output verifies $y(t_0) \geq \delta$. From Proposition~\ref{Proposition : Passage} it follows that $y(t) \geq \delta$ for every time $t\geq0$

Since we are assuming $y(t)\ge \delta$, by Theorem~\ref{Theorem: inverse mapping} and by the fact that $z= T_t(v)$ and $v_0(t)=y(t)$, we have
\begin{equation}
    \label{eq:v-hat-v}
    |v-\hat v| \le |v - \Teinv(z)| + |\Teinv(z) - \Teinv(\hat{z}) | \leq \eta \chi(v) + M_0 |z - \hat{z}|,
\end{equation}
where $M_0>0$ is a uniform bound of the Lipschitz constant of $\mathfrak{T}_t$, which depends on $\eta$ and $\delta$.
We are thus left to estimate $|z-\hat{z}|$. We define $e := \hat z-z$ and consider for any $l>0$ the following Lyapounov function candidate
\begin{equation}
\|e\|_l^2 := e^{\top}P_le,
\end{equation}
where $P_l$ is the matrix with elements $P_{l}(i,j):= \frac{(-1)^{i+j}}{l^{i+j-1}} \cdot \frac{(i + j - 2)!}{(i - 1)!(j - 1)!}$, for $i,j\in\{1,2,3,4\}$.
From  Lemma~\ref{Lemma:Non Linear Part Lipschitz}, we can set $L>0$ to be an upper bound for the Lipschitz constant of $(\D^4 h)_t$ on the compact set $\mathcal{U}:=\{(u_0,u_1,u_2)\in \R^3, \delta \leq |u_0| \leq R \ \text{and} \ \eta \leq|u_{1:2}|\leq R^2\}$, which can be chosen to be uniform with respect to time $t\ge 0$. 
 
Then, since $y = Cz$, it holds
\begin{equation} \label{eq:Error pre gronwal}
    \begin{split}
       \frac{1}{2} \frac{d \|e\|_l ^2}{dt} 
       &=  e^{\top}P_l(A-P_l^{-1}C^{\top}C)e + e^\top P_l W\left( (\D^4 h) \bigl( \mathfrak{T}_t (\hat z) \bigr)- (\D^4 h)( v)\right), \\
       &= -\frac{l}{2} \|e\|_l^2 +  e^\top P_l W\left( (\D^4 h) \bigl( \mathfrak{T}_t (\hat z) \bigr)- (\D^4 h)( v)\right),\\
       &\leq  -\frac{l}{2} \|e\|_l^2 + \sqrt{\frac{\nu_{\max}}{\nu_{\min}}}L M_0 \|e\|_l^2   + \chi(v) L l^{-7/2}{\eta} \|e\|_l .
    \end{split}
\end{equation}
Here, on the second line we used the fact that $P_lA +   A^{\top} P_l - C^{\top}C = -l P_l$. The last line follows by Cauchy-Schwarz inequality and \eqref{eq:v-hat-v}, thanks to the estimates  
\begin{equation}
    \|W\|_l \leq l^{-7/2}\sqrt{\nu_{\max}}
    \qquad\text{and}\qquad 
    l^{-7} \nu_{\min} \operatorname{Id}_4 \le P_l \le l^{-1}\nu_{\max}\operatorname{Id}_4.
\end{equation}
Here, $\nu_{\max}$ and $\nu_{\min}$ are respectively the max and the min eigenvalue of $P_1$.

We denote $l^*:=2\sqrt{\frac{\nu_{\max}}{\nu_{\min}}}L M_0$, and $M:=M_0\frac{\max \{1,\sqrt{\nu_{\max}}\}}{\sqrt{\nu_{\min}}}$. Dividing both sides of \eqref{eq:Error pre gronwal} by $\|e\|_l$, we can apply a linear Grönwall's inequality to the variation of $\|e\|_l$. Since $\|e\|_l\geq \nu_{\min}l^{-7}|\hat z-z|$ and we have \eqref{eq:v-hat-v}, we obtain the following estimation in the immersed domain, for $t\in [t_0,t_f]$,
\begin{equation}  \label{eq:error z domain.} 
|v(t)-\hat{v}(t)| \leq \eta \chi(v(t))  +  M e^{-(l-l^*) t} \biggl( l^3|\hat{z}(t_0)-z(t_0
)| e^{(l-l^*) t_0} + \int_{t_0}^t e^{(l-l^*) s}L \eta \chi(v(s)) \,ds \biggr).
\end{equation}
Using that $z(t)=T_t(v(t))$ for all times $t \geq t_0$ yields the first part of the point.

We now deal with the error increase during the switch.
We recall that $t_1:=\inf \{t \ge 0: |y(t)|<\delta\}=\inf \mathcal{S}$ and $t_2=\sup \mathcal{S}$. We suppose now that $0<t_1<\infty$.
We recall that by Proposition~\ref{Proposition : Passage}, it holds $t_2\leq t_1 + t_\delta$, and  $y(t) \geq \delta$ for all times $t \geq t_2$.
If a switch occurs, the initial condition at $t_1$ for the second dynamics in \eqref{Observer 1} is $\hat v(t_1^+)=\mathfrak{T}_t(\hat{z}(t_1^-))$. The solution $\hat v$ is therefore well defined in the interval $[t_1,t_2]$ and continuous, since $f$ is uniformly Lipschitz and  $y(\cdot)$ is strictly increasing.

Let $L_1>0$ be the Lipschitz constant of $T_t$ on the compact set $\mathcal{U}$, which is uniform with respect to time thanks to Assumption~\ref{Strong Persistence} (see, \textit{e.g.}, proof of Lemma~\ref{Lemma:Non Linear Part Lipschitz}).
Then, since $\hat z(t_2^{-})=T_{t_2}(\hat v(t_2^{-}))$ and $z(t_2)=T_{t_2}(v(t_2))$, we have
\begin{equation}
    \label{eq:A1}
    |\hat{z}(t_2^-)-z(t_2^-)|\le L_1 |\hat{v}(t_2^-)-v(t_2^-)|.
\end{equation}
Letting $L_2>0$ be the Lipschitz constant of $f(\cdot,t)$ on the compact $\mathcal{U}$, which is uniform with respect to time (see, \textit{e.g.}, Proposition~\ref{Solution defined Bounded}), we have
\begin{equation}
    \label{eq:error passage}
    |\hat{v}(t_2^-)-v(t_2^-)| \le e^{(t_2-t_1) L_2} |\hat{v}(t_1)-v(t_1)| 
    \le e^{t_\delta L_2} |\hat{v}(t_1)-v(t_1)|.
\end{equation}
We have therefore characterised a bound of the error for all times $t \geq 0$ using both relations~\eqref{eq:error z domain.} and~\eqref{eq:error passage}.

Let us establish Point \eqref{it:main2} of the theorem. 
Under the assumptions of the statement, $\chi(v(t))=0$ for all times $t \ge 0$. 
Recalling that $\mathcal S=[t_1,t_2]$, we will distinguish two cases: $t_1=0$ or $t_1>0$.
In the first case, by the second part of Point~\eqref{it:main1}, for any $t\le t_2$ we have
\begin{equation}
    |\hat v(t)-v(t)|\le  e^{L_2t_{\delta}}|\hat v(0)-v(0)| \le e^{\left(L_2+(l-l^*)\right)t_{\delta}} e^{-(l-l^*) t}|\hat v(0)-v(0)| .
\end{equation}
Here we used that $e^{-(l-l^*) t_1} \leq e^{-(l-l^*) t} e^{(l-l^*)t_{\delta}}$ as $t \le t_2$. Letting now $t\ge t_2$ and using the first part of Point~\eqref{it:main1} with $t_0=t_2$ we get
\begin{equation}
    |\hat v(t)-v(t)|\le M l^3 e^{\left(L_2+(l-l^*)\right)t_{\delta}} e^{-(l-l^*)t}|\hat v(0)-v(0)|.
\end{equation}
Hence, the statement stands proved in this case, with $k=M l^3 e^{\left(L_2+(l-l^*)\right)t_{\delta}}$.
The argument for the case $t_1>0$ is similar, except that one has to start by applying the first part of Point~\eqref{it:main1} to $t\in [0,t_1]$, then plug this estimate in the error during the switching time for $t\in [t_1,t_2]$, and finally reapply the first part of Point~\eqref{it:main1} to obtain the estimate for $t\ge t_2$.

We finally provide an argument for Point~\eqref{it:main3}. We start by observing that, since $v(0)\in B_{\mathbb{R}^3}(0,R)$ and $\hat v$ is initialized in the compact set $\mathcal K$, it holds $|\hat z(0)-z(0)|\le L_1 (\operatorname{diam}(\mathcal{K}) + R)$. Then, one proceeds exactly as for the previous point except that instead of having $\chi(v(t))=0$ for all $t>0$, one uses the trivial bound $\chi(v(t))\le 1$. 
This leads to the following estimates:
\begin{equation}
   e^{-(l-l^*)t} \int_{0}^{t} e^{(l-l^*)s}\chi(v(s))\,ds \le \frac{1}{l-l^*}
   \quad\text{and}\quad
   e^{-(l-l^*)t} \int_{t_2}^{t} e^{(l-l^*)s}\chi(v(s))\,ds \le \frac{1}{l-l^*}
\end{equation}
We assume $\delta$ to be sufficiently small, to be fixed later, so that $t_\delta<t^*$.
In the case where the first switching time $t_1$ is positive, letting $C>0$ be a constant independent of $\delta$, $\eta$, and $l$ one then obtains that for any $t\in [t^*,t_2]\cup (t_2+t^*,+\infty)$ it holds
\begin{equation}
    \label{eq:dario10}
    \frac{1}{C}|\hat v(t)-v(t)|\le \eta\left(1 + M l^3 e^{L_2t_\delta}e^{-(l-l^*)t^*} \left(1+\frac{M}{l-l^*}\right) + \frac{M}{l-l^*} \right) + M^2 l^6  e^{(L_2+(l-l^*))t_\delta} e^{-(l-l^*)t}.
\end{equation}
Here, one has to be careful due to the fact that $M$ depends on $\eta$, and in particular $M\rightarrow +\infty$ as $\eta\downarrow 0$.
One starts by fixing $\eta<\varepsilon/(4C)$. Then, with $\eta$ fixed, one chooses $l$ sufficiently large so that
\begin{equation}
    M^2 l^6 e^{-(l-l^*)t} \le \frac{\varepsilon}{4C}.
\end{equation}

Assuming that $l-l^*\geq 1$, this also implies that $M/(l-l^*)\leq 1$, so that \eqref{eq:dario10} implies
\begin{equation}
    \label{eq:dario10b}
    |\hat v(t)-v(t)|\le \frac{\varepsilon}{4}\left(2 +  {l^{-3}}e^{L_2 t_\delta}\frac{\varepsilon}{2C}\right) +\frac{\varepsilon}{4}  e^{(L_2+(l-l^*))t_\delta}.
\end{equation}
Then, assuming $l$ is large enough so that $l^{-3}\varepsilon/2C<1$, yields

\begin{equation}
    |\hat v(t)-v(t)|\le {\varepsilon}\left( \frac12+ \frac{e^{(L_2+(l-l^*))t_\delta}}2 \right).
\end{equation}
Finally, since $t_\delta$ tends to $0$ as $\delta$ tends to $0$, we can fix $\delta>0$ such that $e^{(L_2+(l-l^*))t_\delta} \le 1$. This completes the proof in this case.

The case where $t_1=0$ is treated similarly. Here, for any time $t<t_2$ the second part of Point~\eqref{it:main1} only yields a uniform constant bound. However, for any $t>t^*>t_2$, we obtain
\begin{equation}
    \frac{1}{C} |\hat v(t)-v(t)| \le \eta \left( 1 + \frac{M}{l-l^*}\right) + M l^3 e^{(L_2+(l-l^*))t_\delta}e^{-(l-l^*)t}.
\end{equation}
Here, $C>0$ can be taken to be the same constant as in \eqref{eq:dario10}.
It is then easy to see that the choice of $\eta$, $l$, and $\delta$, considered in the case $t_1>0$ imply the statement also in this case. 
\end{proof}

\section{Numerical Simulations} \label{Section:Numerical Simulations}

\begin{figure}[t]
    \centering
\begin{minipage}[c]{.49\textwidth}
\includegraphics[width=\linewidth]{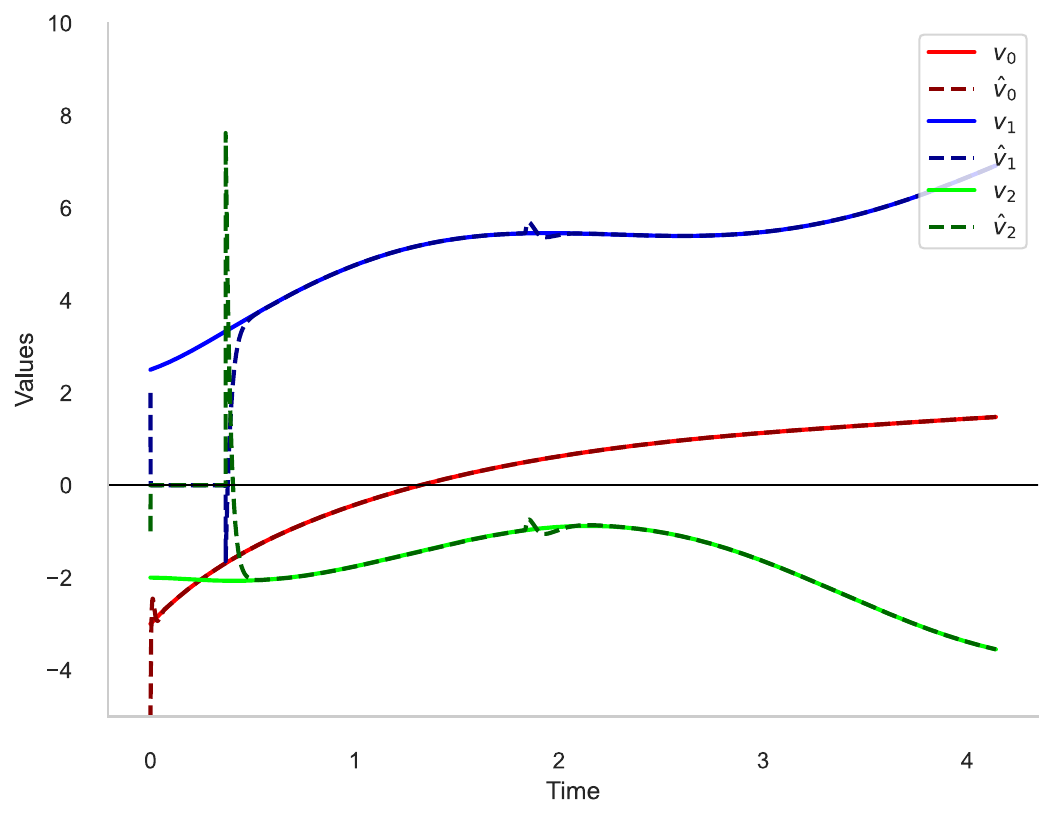}
% \includesvg[width=\linewidth]{PlotObs}
\end{minipage}
\hfill
\begin{minipage}[c]{.49\textwidth}
% \includesvg[width=\linewidth]{Error_Log}
\includegraphics[width=\linewidth]{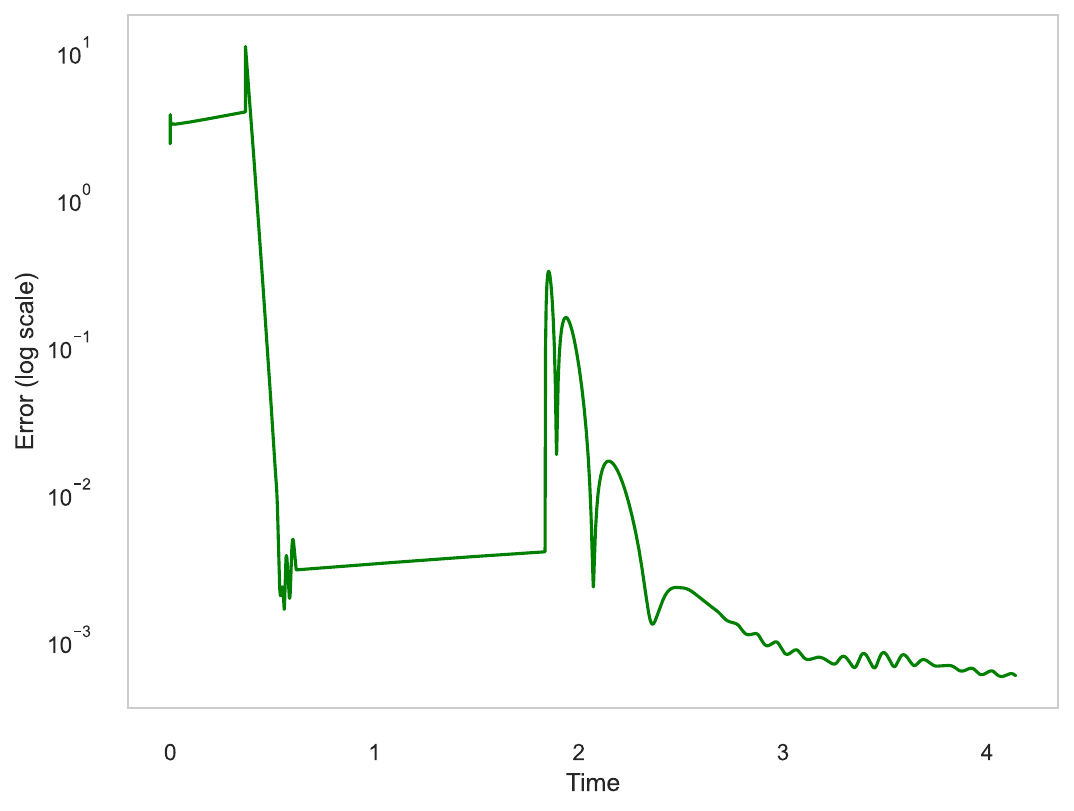}
\end{minipage}
        \caption{Left-hand side: components of the simulated trajectories of System~\eqref{EqP} (plain lines) and Observer~\eqref{Observer 1} (dashed lines). Right-hand side: plot of the estimation log-error.
        This particular solution satisfies point \eqref{it:main2} of the theorem.
        The switching time occur at $t_1\simeq 1.2$ and 
        $t_2\simeq 1.8$.
        Peaking phenomena occur at the start of the simulation and at time $t_2$, when the high-gain observer restarts; they are better seen on the log-error plot.
        Since the error peaking at $t_2$ is proportional to the error at $t_1$, it is much smaller in size. At the beginning of the trajectory, $\hat v_1,\hat v_2$ are set to 0 due to the cut-off \eqref{E:sat}, until the observer $\hat z$ of the embedded system has sufficiently converged. At time $t_1$, the observer is turned off and the trajectories of $v$ and $\hat v$ evolve according to the neural fields dynamics, resulting in a drift of the error over $[t_1,t_2]$.}
        \label{fig:Observer Solutions}
\end{figure}

We propose a numerical simulation for the observer \eqref{Observer 1}. The parameters of the dynamical system~\eqref{EqP} were chosen as such: $J_0=-1$, $J_1=1.5$, $\sigma=\tanh(\mu\cdot- h_0)$ with the non linear gain $\mu=10$ and the threshold $h_0=1$. The distribution $P$ has been chosen as the Dirac mass $P=\delta_{r=1}$ for theoretical purposes.
The input has been chosen as $I_0(t)=\varepsilon(1-\beta)$, $I_{1:2}(t)=\beta \varepsilon \bigl( \cos (\frac{2\pi}{10} t), \sin( \frac{2\pi}{10} t) \bigr)$, with $\beta=10^{-1}$ and $\varepsilon=10^{-1}$. The initial condition was taken as $v(0)=(-3,2.5,-2)$ so that the output verifies $y(0)<-\delta$ to show the switching mechanism.
The timescale has been chosen as $\tau=5$ for theoretical purposes to better show the error increase during the switch. 

The initial condition of the observer was taken as $\hat{v}(0)=(-5,2,-1)$, $\hat{z}(0)=T_0(\hat{v}(0))$. As $J_0<0$, we chose $\delta=3*10^{-1}$, which is greater then the robust bound of Proposition~\ref{Proposition : Passage} (See the remark thereafter). The inversion error was taken $\eta=10^{-3}$, the high gain $l=15$. The numerical scheme chosen was an explicit RK4 with a step of $10^{-5}$.
The resulting trajectory $\hat v$ is shown in Figure~\ref{fig:Observer Solutions} for $t \in [0,4]$.

\appendix

\section{Appendix: Modelisation adjustments} \label{Appendix:Modelisation adjustement.}

Some authors favor a neural field model for activity instead of post-membrane potential, \textit{e.g.}, \cite{ben1995theory,blumenfeld2006neural}. The model is similar to \eqref{GeneralEquation}:
\begin{equation} \label{Activity Model}
    \tau \partial_{t}A(x,t)= \ -A(x,t) \ +  \  \sigma\bigl( J \cdot A(\cdot,t) \  +  I_{\text{ext}}(t) \bigr).
\end{equation}
Supposing $I_{\text{ext}}$ smooth enough, we can easily go from \eqref{Activity Model} to \eqref{GeneralEquation} by considering the change of variable: $V=J\cdot A + I_{\text{ext}},$  and replacing $I_{\text{ext}}$ with $ I_{\text{ext}} + \dot{I}_{\text{ext}}$. We favor \eqref{GeneralEquation} because it appears easier to study from a mathematical point of view.
If our measurement is the averaged neural activity $a_0=\int_{\Omega}A(u) du$, the change of variable $V=J\cdot A + I_{\text{ext}}$ leads to $\int_{\Omega}V(u) du=\langle A,J\cdot1\rangle + \langle I_{\text{ext}},1\rangle=J_0 a_0 + I_0.$
Therefore knowing the parameters $J_0$ and $I_0$, we can retrieve the mean membrane potential $v_0$ based on measurement not done in voltage but in spark per second (activity).

We also have some freedom in changing the shape of the nonlinearity. Some model choose sigmoidal functions $\sigma_+:\R\to \R$ that are strictly positive, strictly increasing (i.e., $\sigma_+'>0$), convex-concave (i.e.,  $x\sigma_+''(x)\leq 0$ for all $x\in \R$), and such that
\[
    \lim\limits_{\substack{x \to -\infty}} \sigma_+(x)=  0, \quad \lim\limits_{\substack{x \to +\infty}} \sigma_+(x)=  1, \quad \sigma_+'>0, \quad \max_{\R} \sigma_+'=\sigma_+'(0).
\]
In this case there exist two parameters $s_1>0$ and $s_2>0$ such that
$\sigma_+=s_1\sigma+s_2$ where $\sigma$ satisfies Assumption~\ref{ass:sigma}. Therefore, it suffices to replace $I_0$ by $I_0 + s_2$ and $J_i$ by $J_i s_1$ for $i\in \{0,1\}$.

On the other hand, if there is a threshold in the sigmoid function, then there exists a constant $h_0$ such that $\sigma (\cdot)$ is replaced by $ \sigma (\cdot -h_0)$. Hence, we can always go back to our main equation by taking the change of variable from $v_0$ to $v_0-h_0$. Replacing then $I_0$ with $I_0 + h_0$ maintains the performed analysis.

\printbibliography

\end{document}